\documentclass[]{siamart0516mod}

\usepackage{amssymb,amsfonts,amsmath}

\usepackage[mathscr]{eucal}

\usepackage{amsbsy}

\usepackage{stmaryrd}

\usepackage{bm}

\usepackage{braket}  

\usepackage{xparse}

\usepackage{xspace}

\usepackage{enumitem}

\usepackage{xcolor}

\usepackage{hyperref}

\usepackage[font=footnotesize,justification=Centering,singlelinecheck=false]{subfig}

\usepackage{cleveref}
\crefformat{footnote}{#2\footnotemark[#1]#3}

\usepackage{pgfplots}
\usepackage{pgfplotstable}
\usetikzlibrary{plotmarks}
\usepgfplotslibrary{colorbrewer}
\usepgfplotslibrary{statistics}
\usepgfplotslibrary{fillbetween}
\pgfplotsset{compat=1.15}

\pgfplotsset{cycle list/Set1-8} %
\pgfplotstableset{col sep=comma}

\usepackage{algorithm, algpseudocode}
\Crefname{ALC@unique}{Line}{Lines}
\algnewcommand{\LineComment}[1]{\State \(\triangleright\) #1}

\NewDocumentCommand \qtext {m} {\quad\text{#1}\quad}

\NewDocumentCommand \Real {} {\mathbb{R}}

\NewDocumentCommand \Natural {} {\mathbb{N}}

\NewDocumentCommand \T { O{} m } {\boldsymbol{#1\mathscr{\MakeUppercase{#2}}}}

\NewDocumentCommand \TCheck { m } {\IfBooleanT{#1}{^{\intercal}}}

\NewDocumentCommand{\Tr}{s}{\IfBooleanTF{#1}{\vphantom{\intercal}}{\intercal}}

\NewDocumentCommand{\Mx}{O{} m !g t' t"}{
  \bm{#1{\mathbf{\MakeUppercase{#2}}}}%
  \IfValueT{#3}{_{#3}}%
  \IfBooleanTF{#4}{^{\Tr}}{%
    \IfBooleanT{#5}{^{\Tr*}}}%
}

\NewDocumentCommand \V { O{} m } {{\bm{#1\mathbf{\MakeLowercase{#2}}}}} 

\NewDocumentCommand \Z { t~ } {\Mx[\IfBooleanT{#1}{\tilde}]{Z}}
\NewDocumentCommand \B { t' t~ t*} {\Mx[\IfBooleanT{#2}{\tilde}]{B}\IfBooleanT{#1}{^{\intercal}}\IfBooleanT{#3}{_{\ast}}}
\NewDocumentCommand \A { } {\Mx{A}}
\NewDocumentCommand \Xmat { t' t~} {\Mx[\IfBooleanT{#2}{\tilde}]{X}\IfBooleanT{#1}{^{\Tr}}}
\NewDocumentCommand \Om { t! } {\Mx[\IfBooleanT{#1}{\bar}]{\Omega}}
\NewDocumentCommand \bvec { t~ t* } {\V[\IfBooleanT{#1}{\tilde}]{\alpha}\IfBooleanT{#2}{_{\ast}}}

\NewDocumentCommand \Vp { } {\V{p}}
\NewDocumentCommand \Multi { } {\text{\sc multinomial}}
\NewDocumentCommand \RandSample { } {\text{\sc randsample}}

\NewDocumentCommand \Prob { } {\text{\rm Pr}}
\NewDocumentCommand \Exp { } {\mathbb{E}}

\NewDocumentCommand \levZ { s } {\IfBooleanT{#1}{\bar}\ell_i(\Z)}
\NewDocumentCommand \levAk { } {\ell_{i_k}(\Ak)}

\NewDocumentCommand \Vl { } {\V{\ell}}

\DeclareDocumentCommand \det { } {\text{\rm\sffamily det}}
\NewDocumentCommand \rnd { } {\text{\rm\sffamily rnd}}

\NewDocumentCommand \sdet { } {s_{\text{\rm\sffamily det}}} 
\NewDocumentCommand \srnd {s } {\IfBooleanT{#1}{\bar}s_{\text{\rm\sffamily rnd}}} 
\NewDocumentCommand \pdet { } {p_{\text{\rm\sffamily det}}} 

\NewDocumentCommand{\pk}{G{k}}{\V{p}_{#1}}
\NewDocumentCommand{\pke}{G{k} G{i_k}}{(\pk{#1})_{#2}}

\NewDocumentCommand \DetSet { } {\mathcal{D}} %
\NewDocumentCommand \DetSetk { G{k} } {\mathcal{\bar D}_{#1}}

\NewDocumentCommand{\resid}{}{\mathcal{R}}

\NewDocumentCommand{\DIDX}{}{\textsc{DetSkrp}}
\NewDocumentCommand{\SIDX}{}{\textsc{RndSkrp}}
\NewDocumentCommand{\CIDX}{}{\textsc{CombineRepeats}}

\NewDocumentCommand{\idxdet}{}{\texttt{idet}}
\NewDocumentCommand{\idxrnd}{}{\texttt{irnd}}
\NewDocumentCommand{\wgtrnd}{}{\texttt{wrnd}}
\NewDocumentCommand{\idx}{}{\texttt{idx}}
\NewDocumentCommand{\wgt}{}{\texttt{wgt}}

\NewDocumentCommand{\cpals}{}{{CP-ALS}\xspace}
\NewDocumentCommand{\cprand}{}{{CP-ARLS}\xspace}
\NewDocumentCommand{\cprandlev}{}{{CP-ARLS-LEV}\xspace}

\NewDocumentCommand \X { } {\T{X}}

\NewDocumentCommand \Xk { G{k} } {\Mx{X}_{(#1)}}

\NewDocumentCommand \M {} {\T{M}}

\NewDocumentCommand \Mk { O{k} } {\Mx{M}_{(#1)}}
\DeclareDocumentCommand \mi { s } 
{
  \IfBooleanTF{#1}
  {m(i_1,i_2\dots,i_d)}
  {m_{i}}
}

\NewDocumentCommand \Ak { G{k} t' t"  } { \Mx{A}_{#1}\IfBooleanTF{#2}{^{\intercal}}{}\IfBooleanTF{#3}{^{\phantom{\intercal}}}{} }

\NewDocumentCommand \AkAkt { G{k} } {\Ak{#1}'\Ak{#1}"}

\NewDocumentCommand \Ake { G{k} G{i} G{j} } {
  a_{#1}(#2_{#1},#3)
}

\NewDocumentCommand \Zk { G{k} t' t"} {\Mx{Z}_{#1}\IfBooleanTF{#2}{^{\intercal}}{}%
  \IfBooleanTF{#3}{^{\phantom{\intercal}}}{}}

\NewDocumentCommand \lvec {} {\V{\lambda}}

\NewDocumentCommand \lj { G{j} } {\lambda_{#1}}

\NewDocumentCommand \KT { s G{d+1}} {
  \llbracket 
  \IfBooleanTF{#1}{\lvec;}{}
  \Ak{1}, \Ak{2}, \dots,  \Ak{#2} \rrbracket
}

\DeclareMathOperator{\nnz}{nnz}

\newcommand{\trans}{\ensuremath{\mathsf{T}}}
\DeclareMathOperator*{\argmin}{arg\,min}

\NewDocumentCommand \Akj {O{k} G{j}} {\V{a}_{#1}(:,#2)}

\DeclareMathOperator{\rank}{rank}

\newcommand{\R}{\mathbb{R}}
\newcommand{\E}{\mathbb{E}}

\newcommand{\Bperp}{\ensuremath{\Mx{B}^{\perp}}}
\newcommand{\UA}{\ensuremath{\Mx{U}_{\Mx{A}}}}

\newcommand{\Xopt}{\ensuremath{\Mx{X}_{\text{\rm\sffamily opt}}}}
\newcommand{\Xtildeopt}{\ensuremath{\widetilde{\Mx{X}}_{\text{\rm\sffamily opt}}}}

\newcommand{\sfit}{s_{\text{fit}}}

\NewDocumentCommand{\appproof}{}{\cref{sec:proof}}
\NewDocumentCommand{\appcomplexity}{}{\cref{sec:complexity}}
\NewDocumentCommand{\appsampling}{}{\cref{sec:sampling}}
\NewDocumentCommand{\appdense}{}{\cref{sec:dense}}
\NewDocumentCommand{\appuber}{}{\cref{sec:uber-details}}
\NewDocumentCommand{\apprrf}{}{\cref{sec:rrf}}
\NewDocumentCommand{\appreddit}{}{an ancillary file\footnote{See the ancillary file titled ``Reddit Factors.''}}

\author{%
  Brett W. Larsen%
  \thanks{Stanford University, Stanford, CA (\email{bwlarsen@stanford.edu})}
  \and Tamara~G.~Kolda%
  \thanks{MathSci.ai, Dublin, CA (\email{tgkolda@mathsci.ai})}
 }

\title{Practical Leverage-Based Sampling for Low-Rank Tensor Decomposition%
  \thanks{This work has been supported in part by the Department of Energy
    Office of Science Advanced Scientific Computing Research Applied Mathematics Program.
    Additionally, the work of BL was supported in part by the Department of Energy
    Computational Sciences Graduate Fellowship program (DE-FG02-97ER25308).}}

\headers{Leverage-Based Sampling for Tensor Decomposition}
{Brett W. Larsen and Tamara G. Kolda}

\begin{document}

\maketitle

\begin{abstract} 
  The low-rank canonical polyadic tensor decomposition is useful in data analysis and can be computed by
  solving a sequence of overdetermined least squares subproblems.
  Motivated by consideration of sparse tensors,
  we propose sketching each subproblem using leverage scores to select a subset of the rows, 
  with probabilistic guarantees on the solution accuracy.
  We randomly sample rows proportional to leverage score upper bounds that can be efficiently computed
  using the special Khatri-Rao subproblem structure inherent in tensor decomposition.
  Crucially,  for a $(d+1)$-way tensor, the number of rows in the sketched system is $O(r^d/\epsilon)$ for a decomposition of rank $r$ and $\epsilon$-accuracy in the least squares solve, independent of
  both the size and the number of nonzeros in the tensor.
  Along the way, we provide a practical solution to the generic matrix sketching problem of sampling overabundance
  for high-leverage-score rows,
  proposing to include such rows deterministically and
  combine repeated samples in the sketched system;
  we conjecture that this can lead to improved theoretical bounds.
  Numerical results on real-world large-scale tensors
  show the method is significantly faster than deterministic
  methods at nearly the same level of accuracy.
\end{abstract}


\section{Introduction}
\label{sec:introduction}

Low-rank canonical polyadic or CANDECOMP/PARAFAC (CP) tensor decomposition \cite{CaCh70,Ha70},
is a popular unsupervised learning method
akin to low-rank matrix decomposition and principal component analysis (PCA).
A low-rank tensor factorization identifies \emph{factor matrices} that
provide the best low-rank multilinear representation of a higher-order tensor.
Tensor decomposition
is ubiquitous in data analysis
with applications to
social networks~\cite{nakatsuji2017semantic, papalexakis2015location},
ride sharing~\cite{yan2018visual},
cyber security~\cite{maruhashi2011},
criminology~\cite{mu2011empirical},
text clustering~\cite{drakopoulos2017tensor},
online behaviors~\cite{sapienza2018non},
etc.
We refer the reader to several surveys~\cite{AcYe09,KoBa09,SiDeFuHu16} for more information.

In this work, we consider the problem of computing the CP tensor decomposition
for sparse tensors using an alternating least squares (ALS) approach. 
Bader and Kolda~\cite{BaKo07} show that the cost per least squares solve
for a sparse tensor is proportional to the number of nonzeros.
However, in many cases, even that can be too expensive because some tensors have billions of nonzeros.
Cheng et al.~\cite{ChPePeLi16} showed that it is possible to use matrix sketching in the three-way sparse case.
We propose a different application of matrix sketching, targeting a different step in the least squares solve (explained in detail below), achieving
an improved sampling bound.
In addition, we present a detailed, practical algorithm along with new methodologies for handling the overabundance of high-probability rows based on combining repeat rows and a
hybrid strategy for combining deterministic and random samples. %

\subsection{CP  least squares problem}
We focus immediately on the prototypical least squares problem, deferring detailed definitions and
derivations until \cref{sec:backgr-least-squar}.
Let  $\X$ be a $(d{+}1)$-way tensor of size $n_1 \times n_2 \times \cdots \times n_{d{+}1}$.
The goal of CP is to compute $d+1$ factors matrices, denoted $\set{\Ak{1},\dots,\Ak{d+1}}$, that
can be assembled into a low-rank model of $\X$.
Each iteration of \cpals solves a sequence of $(d{+}1)$ least squares problems.
Without loss of generality, we consider the least squares problem for computing the $(d{+}1)$st factor matrix
with $\set{\Ak{1},\dots,\Ak{d}}$  fixed:
\begin{equation}
  \label{eq:lsq_krp}
  \begin{gathered}
    \min_{\B}  \| \Z \B' - \Xmat' \|_F^2
    \qtext{subject to}
    \B \in \Real^{n \times r}
  \qtext{with}
  \\
  \Z = \Ak{d} \odot \cdots \odot \Ak{1} \in \Real^{N \times r},
  \quad
  \Ak \in \Real^{n_k \times r}
  \text{ for } k \in [d],
  \quad
   N = \prod_{k=1}^d n_k, \\
  \Xmat \in \Real^{n \times N},
  \qtext{and}
  r,n \ll N.
  \end{gathered}
\end{equation}
The symbol $\odot$ denotes the Khatri-Rao product (KRP); see \cref{sec:backgr-least-squar}.
The matrix $\Xmat$ is the mode-$(d{+}1)$ unfolding of the input tensor.
The matrix $\B$ is the $(d{+}1)$st factor matrix and $n=n_{d{+}1}$.
If $\X$ is sparse, then $ \nnz(\Xmat) \ll N$,
and sparse tensors are the primary focus of our work.

Because $r \ll N$, the least squares problem \cref{eq:lsq_krp} is tall and skinny, making it a candidate for sketching.
Ignoring the structure of both $\Z$ and $\Xmat$, solving \cref{eq:lsq_krp} costs
$O(Nnr + Nr^2)$
floating point operations (FLOPS) as the QR decomposition of $\Z$ is computed once for a cost of $O(Nr^2)$ and then applied $n$ times.
We can alternatively exploit the KRP structure of $\Z$ to compute the solution at the same cost of $O(N nr)$
via the normal equations~
\cite{KoBa09}.
A main advantage of the latter approach is that the cost reduces to $O(\nnz(\X)\,r)$ when $\X$ is sparse \cite{BaKo07}.

Instead of solving the least squares problem \cref{eq:lsq_krp} directly,
we consider a \emph{sketched} version of the form
\begin{equation}
  \label{eq:sketched_lsq}
  \min_{\B}  \| \Om \Z \B' - \Om \Xmat' \|_F^2,
  \qtext{where} \Om \in \Real^{s \times N}
\end{equation}
and $\Om$ has only one nonzero per row, which means that it
selects a subset of rows in the least squares problem.
The cost of solving the subsampled least squares problem is
$O(snr + sr^2)$, where we assume $N \gg s > \max\set{n,r}$.
In the case where $\Xmat$ is sparse, the $O(snr)$ complexity can be be reduced by using sparse-matrix multiplication.
If we sample rows proportional to the products of the leverage scores (see \cref{def:leverages_scores}) of the constituent factor matrix rows, then
$s=O(r^{d}/ \epsilon)$ rows are required for an $\epsilon$-accurate solution with high probability and for $\epsilon$ sufficiently small; see \cref{thm:krp-beta}.
Moreover, we compute $\Z~ \equiv \Om\Z$ and $\Xmat'~ \equiv \Om\Xmat'$ without every forming $\Om$, $\Z$, or $\Xmat$ explicitly.
This means the complexity of the sketched least squares problem is 
$O(\max\set{n,r}r^{d+1} / \epsilon)$ for a $(d+1)$-way tensor,
with a further reduction in the sparse case possible using sparse-matrix multiplication.

\subsection{Related work}

A variety of randomized algorithms have been applied to the CP decompositions in previous work.
A Kronecker fast Johnson-Lindenstrauss transform (KFJLT) sketching approach was proposed
in \cite{BaBaKo18} and proven to be a Johnson-Lindenstrauss transform in \cite{JiKoWa20} (see also \cite{MaBe20,IwNeReZa19}).
The KFJLT reduces the per-iteration cost to $O(s n r + sr^2)$ where $s \ll N$ is the number of samples.
Unfortunately, the KFJLT approach is not applicable in the sparse case.
Even when the tensor $\X$ is sparse, this only makes the $\Xmat'$ in the least squares problem sparse,
and the matrix $\Z$ is still dense because it is a (dense) KRP.
Forming $\Z$ would require $O(Nr)$ flops and storage before it could be multiplied by a sketching matrix.
The KFJLT approach avoids forming $\Z$; instead, it separately multiplies each factor matrix ($\Ak$)
in the KRP by an FJLT and then only forms those rows of $\Z$ that are in the sketch.
Importantly, this approach necessitates special preprocessing in the form of multiplying $\X$ by an FFT in each mode, destroying sparsity if it exists.
For instance, the smallest sparse tensor we use in our experiments
(the ``Uber'' tensor) needs only 0.14~GB of storage for its sparse representation,
but the KFJLT-preprocessed dense version would need 68~GB of storage.
This is our motivation for the alternative to the KFJLT.
Nevertheless, although this work focuses on sparse tensors, our proposed leverage score sampling is a viable approach for dense tensors as we show numerically in \appdense.
According to  \cite{JiKoWa20}, the number of samples required for the KFJLT is $s=O(\log^{2d-1}r \log N / \epsilon^2)$, so our sample complexity is also improved.

Also for dense tensors, one approach to the CP decomposition is to randomly compress the tensor and then decompose the compressed version.
Zhou, Cichocki, and Xie~\cite{ZhCiXi14} used the randomized range finder to estimate the basis for the Tucker decomposition and then compute
the CP decomposition of the Tucker core.
ParaSketch by Yang, Zamzam, and Sidiropoulos~\cite{yang2018parasketch} uses sketching to parallelize the algorithm; a small sketch of the tensor is independently constructed multiple times, decomposed in parallel, and combined.
An alternative approach proposed by Wang, Tung, Smola, and Anandkumar~\cite{wang2015fast} is to construct a sketch of the tensor (called TensorSketch in analogy to CountSketch) and then use this sketch to compute the MTTKRP; 
the number of samples required for general tensors, however, is not derived.
Randomized algorithms have also been applied to decompositions other than CP, including the Tucker decomposition~\cite{DrMa07,malik2018low,ZhLuRaYu20}, the Tucker decomposition with streaming data~\cite{sun2019low, ahmadi2020randomized}, and the CUR decomposition~\cite{mahoney2008tensor,SoWoZh19}.

The most similar work to ours is Cheng et al.~\cite{ChPePeLi16},
described briefly here.
As mentioned above, the KRP structure of $\Z$ can be exploited
to translate \cref{eq:lsq_krp} into an equivalent least squares problem of
the form
\begin{displaymath}
  \min_{\B} \| \B \Mx{V} - \Z' \Xmat \|_F^2
  \qtext{where}
  \Mx{V} = (\Ak{1}'\Ak{1}) \circledast \cdots \circledast (\Ak{d}'\Ak{d}) \in \Real^{r \times r},
\end{displaymath}
where $\circledast$ represents the Hadamard product.
The difference is that Cheng~et~al.\ sketch the matrix product $\Z'\Xmat$ rather than the full least squares problem,
and as a result, our theory and sampling schemes differ.
For example, the theory of Cheng~et~al.\ requires samples based on rows of
the unfolded tensor $\Xmat$ \emph{and} leverage scores
of the factor matrices that comprise $\Z$,
where we use only the latter and our sampling has no direct dependency
on the rows of $\Xmat$.
They only prove their result for third-order tensors,
but their result should generalize to $s=O(r^{d}\log n/\epsilon^2)$
for the $(d+1)$-way case.
Our proposed approach uses fewer samples since we have $s=O(r^d/\epsilon)$.

After the initial version of this work was posted, Aggour, Gittens, and B{\"u}lent~\cite{AgGiYe20} proposed a sketching framework for CP
that chooses the number of samples per sketch in order to guarantee
high-accuracy solutions and a corresponding reduction in the norm at each step.
They develop a methodology to adjust the number of samples over time to ensure the appropriate accuracy in the solutions,
and this method requires computing some full solutions in order to validate the accuracy and determine when
the number of samples needs to be increased. 
We do not compare to this method because it is not applicable to the large-scale sparse tensors
under consideration here. They focus mainly on dense tensors and consider only one sparse tensor of size
$120 \times 918 \times 2{,}881$ with approximately 85,000 nonzeros.
They report times over 300 seconds for their method applied to this sparse problem, but they do not
compare to the standard CP-ALS. For a sparse tensor of that size and density, the standard
CP-ALS requires approximately 0.1~seconds per iteration on a laptop
computer\footnote{Intel(R) Core(TM) i7-10510U CPU \@ 1.80~GHz with 16~GB RAM running MATLAB}, and no more than
O(10~sec.) overall to solve such a problem. Such small problems do not require sketching as we propose here since the standard CP-ALS
method is already extremely fast, i.e., $O(\nnz(\X)r)$ operations per least squares solve.

\subsection{Our contributions}

Randomized numerical linear algebra has the potential to significantly accelerate the solution to the least squares subproblems in \cpals.
In the sparse case, we would ideally sample rows in the least squares problem according to leverage scores. We cannot calculate
the leverage scores directly, but we can instead upper bound them using the structure of the KRP and efficiently
sample proportional to these bounds. Our contributions are as follows:
\begin{itemize}[leftmargin=0.5cm]
\item We detail CP-ARLS-LEV, a practical method for leverage-score randomized sampling in the context of \cpals; see \cref{sec:full-algorithm}.
\item We prove that the proposed method obtains an $\epsilon$-accurate solution with high probability
  with $s=O(r^{d} / \epsilon)$ rows for suitably small epsilon; see \cref{thm:krp-beta}.
  The number of samples is independent of the number of rows in the system and the number of nonzeros in the data tensor.
\item For concentrated sampling probabilities that result in an overabundance of high-probability rows, we propose two novel methods in \cref{sec:tools-sketching-with}:
  (1) combining repeated rows, and (2) including high-probability rows deterministically. These methods can be used in any sketching scenario, not just for tensor decomposition.
\item We present detailed numerical experiments in MATLAB showing the advantages of our proposed approach in \cref{sec:numerical-results}.
  For instance,
  compared to \cpals,
  we achieve a speed-up of up to 16 times on the large Reddit tensor which has 4.7 billion nonzeros, reducing the compute time from about 4 days to 6 hours.
\end{itemize}

\section{Background on least squares problems in \cpals and KRPs}
\label{sec:backgr-least-squar}

\Cref{eq:lsq_krp} represents the prototypical least squares problem in \cpals.
We have assumed that we are solving the $(d{+}1)$st subproblem for notational convenience,
but all $(d{+}1)$ subproblems have precisely the same format.
For instance, if we were solving the least squares subproblem for first factor matrix ($\Ak{1}$),
then \cref{eq:lsq_krp} would change only in that $n=n_1$, $\Xmat$ is the mode-1 unfolding of $\X$, and
$\Z = \Ak{d+1} \odot \cdots \odot \Ak{2} \in \Real^{N \times r}$ with $N = \prod_{k=2}^{d+1} n_k$.
Henceforth, without loss of generality, we continue to assume that we are solving the $(d{+}1)$st subproblem as
in \cref{eq:lsq_krp}.
For simplicity of discussion, we assume throughout that $\rank(\Z) = \rank(\Ak{1}) = \cdots =  \rank(\Ak{d}) = r$.

The KRP plays a key role in our discussion, so we provide a precise definition.
If $\Z = \Ak{d} \odot \cdots \odot \Ak{1}$, then
there is a bijective mapping between row $i$ of $\Z$ and a $d$-tuple of rows $(i_1,\dots,i_d)$ in the factor matrices
where
\begin{equation}
  \label{eq:Zrow}
  \Z(i,:) = \Ak{1}(i_1,:) \circledast \cdots \circledast \Ak{d}(i_d,:).
\end{equation}
Specifically,
we refer to
$(i_1,\dots,i_d) \in [n_1] \otimes \cdots \otimes [n_d]$ as the \emph{multi-index}
and $i \in [N]$ as the \emph{linear index} where the bijective mapping is
\begin{equation}\label{eq:bijection}
  i = i_1 + \sum_{k=2}^d \left(\prod_{\ell=1}^{k-1} n_k\right) (i_k-1).
\end{equation}

\section{Background on sketching for least squares problems}
\label{sec:backgr-sketch-least}
For detailed information on leverage score sampling in matrix sketching, we refer the reader to the surveys \cite{Ma11,Wo14}.
Here we provide key concepts that are needed in this work.

Our goal is to find a sampling matrix $\Om$ so that $\Om \Xmat$ can be computed efficiently when $\Xmat$ is sparse.
To accomplish this, we limit our attention to choices for $\Om$ where each row has a single nonzero.
As discussed in the introduction, solving the least squares problem \cref{eq:lsq_krp} directly costs $O(N r^2 + Nnr)$.
Our goal is to eliminate dependence on $N$.
In this section, we review the theory which explains how to reduce
the cost to $O(s r^2 + snr)$ where $s$ depends in part on how we perform the sampling (the $O(snr)$ complexity can be further reduced by exploiting the sparsity of $\Xmat$).
This removes one dependence on $N$, and
in \cref{sec:effic-matr-sketch}, we explain how to avoid explicitly forming the KRP or
calculating the leverage scores, removing the remaining dependence on $N$.

\subsection{Weighted sampling}
\label{sec:weighted-sampling}

Assuming we choose rows according to some probability distribution, we show how to weight the rows
so that the subsampled norm is unbiased.

\begin{definition}
We say $\Vp \in [0,1]^N$ is a \emph{probability distribution} if 
$\sum_{i=1}^N p_i = 1$.  
\end{definition}

\begin{definition}
  For a random variable $\xi \in [N]$, we say $\xi \sim \Multi(\Vp)$ if $\Vp \in [0,1]^N$ is a probability distribution and $\Prob(\xi = i) = p_i$.
\end{definition}

We can define a matrix that randomly samples rows from a matrix (or elements from a vector) with weights as follows.
The following definition can be found, e.g., in \cite[Defn.~16]{Wo14} or \cite[Alg.~1]{DrMa17}.

\begin{definition}\label{def:randsample}
  We say $\Om \in \Real^{s \times N} \sim \RandSample(s,\Vp)$
  if $s \in \Natural$,
  $\Vp \in [0,1]^N$ is a probability distribution,  
  and the entries on $\Om$ are defined as follows.
  Let $\xi_j \sim \Multi(\Vp)$ for $j=1,\dots,s$; then
  \begin{displaymath}
    \omega(j,i) =
    \begin{cases}
      \frac{1}{\sqrt{s p_i}} & \text{if } \xi_j = i, \\
      0 & \text{otherwise},
    \end{cases}
    \qtext{for all}
    (j,i) \in [s] \times [N].
  \end{displaymath}
\end{definition}

It is straightforward to show that such a sampling matrix is unbiased, so we leave the proof of the next lemma as an exercise for the reader.

\begin{lemma}\label{lem:expectation}
  Let $\V{x} \in \Real^N$.
  Let $\Vp \in [0,1]^N$ be probability distribution such that $p_i > 0$ if $x_i \neq 0$
  and let $\Om \sim \RandSample(s,\Vp)$.
  Then $\Exp{ \| \Om \V{x} \|_2^2} = \|\V{x}\|_2^2$.
\end{lemma}

The challenge of sketching is to design a sampling matrix $\Om$ that can be efficiently computed
yet bounds the distortion of the sketched solution with as few samples as possible.
There is a vast literature on different methods for constructing sketches, but here we focus on row sampling in which a sketch provides a procedure for how to select and weight $s$ rows of the original matrix. Doing this effectively often requires an understanding of the structure of the data, and to that end, we define the leverage scores of a matrix in the next subsection.

\subsection{Leverage scores and sampling probabilities}
\label{sec:leverage-scores}

The distribution selected for $\Vp$ determines the quality of the estimate in a way that depends on the leverage scores of $\Z$.

\begin{definition}[Leverage Scores \cite{DrMaMaWo12}]\label{def:leverages_scores}
  Let $\Mx{Z} \in \Real^{N \times r}$ with $N > r$, and
  let $\Mx{Q} \in \Real^{N \times r}$ be any orthogonal basis for the column space of $\Z$.
  The \emph{leverage scores} of the rows of $\Z$ are given by
  \begin{displaymath}
    \ell_i(\Z) = \|\Mx{Q}(i,:)\|_2^2 \qtext{for all} i \in \set{1,\dots,N}.
  \end{displaymath}
  The \emph{coherence} is the maximum leverage score, denoted
  $\mu(\Z) = \max_{i \in [N]} \ell_i(\Z)$.
\end{definition}

The leverage scores indicate the relative importance of rows in the matrix $\Z$.
It is known that $\ell_i(\Z) \leq 1$ for all $i \in [N]$, $\sum_{i\in[N]} \ell_i(\Z) = r$,  and $\mu(\Z) \in [r/N,1]$ \cite{Wo14}.
The matrix $\Z$ is called incoherent if $\mu(\Z) \approx r/N$.

If we sample based on leverage scores, a sketched system can achieve an $\epsilon$-accurate solution, as stated in the following theorem.
Specifically, the number of samples required to obtain
an $\epsilon$-accurate solution
with probability $1-\delta$ is 
\begin{displaymath}
  s=({r}/{\beta}) \max \set{ C \log (r/\delta), {1}/(\delta \epsilon) }
\end{displaymath}
where $C$ is a constant.
The quantity $\beta$ connects the leverage scores and the sampling probabilities,
and $\beta$ is  sometimes referred to as the \emph{misestimation} factor.
We generally treat $\delta$ as a constant and assume $\epsilon$ is sufficiently small so that $\epsilon^{-1} \geq C \delta \log(r/\delta)$. In this case, we can
write $s=O(r/(\beta \epsilon))$.

\begin{theorem}%
  \label{thm:sketching}
  Consider the least squares problem
  $\min_{\B \in \Real^{r \times n}} \| \Z \B' - \Xmat'\|^2$
  where $\Z \in \Real^{N \times r}$ with $r \ll N$, $\rank(\Z) = r$,
  and  $\Xmat \in \Real^{n \times N}$.
  Let $\Vp \in [0,1]^N$ be a probability distribution and
  assume there exists a fixed
  $\beta \in (0,1]$ such that
  \begin{displaymath}
    \beta \leq \min_{i \in [N]} \frac{p_i r}{\ell_i(\Z)} \qtext{for all} i \in [N].
  \end{displaymath}  
  For any  $\epsilon, \delta \in (0,1)$,
  set 
  $s=({r}/{\beta}) \max \set{ C \log (r/\delta), {1}/(\delta \epsilon) }$
  where $C = 144/(1-1/\sqrt{2})^2$
  and let $\Om = \RandSample(s,\Vp)$.
  Define 
  $\B* \equiv \arg \min_{\B \in \Real^{r \times n}} \| \Z \B' - \Xmat'\|^2$.
  Then $\B~* \equiv \arg \min_{\B \in \Real^{r \times n}} \| \Om \Z \B' - \Om \Xmat'\|_F^2$
  satisfies
  \begin{displaymath}
    \| \Z \B'~* - \Xmat' \|_F^2
    \leq (1+\epsilon) \| \Z \B'* - \Xmat' \|_F^2
  \end{displaymath}
  with probability at least $1-\delta$.
\end{theorem}

To the best of our knowledge, this precise result is new so we provide its proof in \appproof.
From this result, we can see that
the user should ideally specify $\Vp$ so that $\beta$ is maximal, i.e., $p_i = \ell_i(\Z)/r$ for all $i \in [N]$ would yield $\beta=1$.
But computing the true leverage bounds is too expensive.
Instead, we estimate them and get a bound that yields $\beta = 1/r^{d}$ as explained in \cref{sec:leverage-score-upper}.

\section{Tools for sketching with concentrated sampling probabilities}
\label{sec:tools-sketching-with}

In this section, we discuss two novel approaches to improve the computational cost of
sketching for matrices with concentrated sampling probabilities, i.e.,
a small subset of the
rows accounts for a significant portion of the probability mass.
In these cases, a small subset of rows are repeatedly re-sampled
which leads to a larger number of required samples and
is inefficient. 
In \cref{sec:comb-repe-rows},
we show that one simple speedup is to combine (and appropriately
reweight) repeated rows, reducing the size of the sampled least
squares problem without changing the solution.
In \cref{sec:comb-determ-rand}, we propose a novel hybrid sampling method
in which we deterministically include a relatively small number of high-probability rows
and then sample randomly from the remaining rows.

These results assume no special structure in the least squares problem.
Additionally,
they can be implemented in an efficient solver for arbitrary sampling
probability distributions, i.e., they do not require a priori knowledge that
the probabilities are concentrated. If the probabilities are close to
uniform, then the solver is essentially unchanged.  This is
crucial in the case of solving a series of least squares problems which may
each have different characteristics. In our case for the CP tensor factorization, the factor matrices are initialized
randomly (with near-uniform sampling probabilities) and
often have much more structured factored matrices (with concentrated sampling probabilities) as the method converges.
We show numerical improvements yielded by these methods 
in 
\cref{sec:exp-comb-repeat,sec:hybrid-rand}.

\subsection{Combine repeated rows}
\label{sec:comb-repe-rows}
If the random sampling of a matrix selects the same rows repeatedly, it is possible to combine repeated entries.
This results in a smaller matrix that yields an equivalent sampled system.
Consider the definition of $\Om$ in \cref{def:randsample}.
Let $\bar s$ be the number of unique values in the set $\Xi \equiv \set{\xi_1, \dots, \xi_s}$,
let $\bar\xi_j$ denote the $k$th unique value for $k \in [\bar s]$, and let $c_j$ be the number of times
that $\bar\xi_j$ appeared in $\Xi$.
Define $\Om! \in \Real^{\bar s \times N}$ as follows:
\begin{equation}\label{eq:combo}
  \bar \omega(j,i) =
  \begin{cases}
    \sqrt{ \frac{c_j}{s p_i} } & \text{if } \bar \xi_j = i \\
    0 & \text{otherwise}    
  \end{cases}
    \qtext{for all}
    (j,i) \in [\bar s] \times [N].
\end{equation}
It can be shown that $\| \Om! \V{x} \|_2 = \| \Om \V{x} \|_2$ for all $\V{x} \in \Real^N$.
Unless otherwise noted, we combine repeated rows in our experiments.

\subsection{Hybrid deterministic and random sampling}
\label{sec:comb-determ-rand}

One potential alternative to probability sampling is to sort by descending probability 
and deterministically construct a matrix sketch using the top $s$ rows.
For instance  Papailiopoulos, Kyrillidis, and Boutsidis \cite{papailiopoulos14} theoretically analyzed
the quality of such approximations and show they perform comparably to a randomized approach
if the leverage scores fall off according to a moderately steep power law.

In this section we propose a more flexible alternative in which a subset of the highest
probability rows are included deterministically and the remaining rows are chosen
randomly, proportional to their original probabilities.
Hybrid methods combining deterministic and randomized methods have also been considered in the context of the CUR decomposition \cite{martinsson2020randomized}.
We combine deterministic and random sampling for KRP matrices by constructing
a sampling matrix of the following form:
\begin{displaymath}\label{eq:Om-hybrid}
  \Om =
  \begin{bmatrix}
    \Om_{\det} \\
    \Om_{\rnd}
  \end{bmatrix}
  \in \Real^{s \times N}.
\end{displaymath}
Note that this matrix is never actually formed explicitly, as detailed in \cref{sec:altern-rand-least}.

Let $\DetSet \subset [N]$ be the set of indices that are included deterministically with $\sdet = |\DetSet|$.
We presume that $\DetSet$ contains the highest-probability indices
which would be more likely to be repeated randomly,
but our analysis regarding the reweighting of the remainder does not depend on this.
In particular, it still applies if only a subset of the highest probability rows
are included.
Let $k_j$ denote the $j$th member of $\DetSet$, $j \in [\sdet]$.
Then we have the corresponding deterministic row sampling matrix
\begin{equation}
  \label{eq:Om-det}
  \omega_{\det} (j,i) =
  \begin{cases}
    1 & \text{if } i = k_j \\
    0 & \text{otherwise}
  \end{cases}
  \qtext{for all}
  (j,i) \in [s_{\det}] \times [N].
\end{equation}

We randomly sample the remaining rows from $[N] \setminus \DetSet$. Define $\pdet = \sum_{i \in \DetSet} p_i$.
The probability of selecting item $i \in [N] \setminus \DetSet$ is rescaled to $p_i / (1-\pdet)$.
(We do not compute these explicitly, as detailed in \cref{sec:impl-determ-row}.)
Then we have
\begin{displaymath}\label{eq:Om-rnd}
  \omega_{\rnd}(j,i) =
  \begin{cases}
    \sqrt{ \frac{ 1-\pdet}{s p_i} } & \text{if } \xi_j = i \\
    0 & \text{otherwise}.
  \end{cases}
\end{displaymath}

\subsection{Combining rows for the hybrid deterministic and random sampling}
\label{sec:comb-rows-hybr}

We can also combine repeated rows in $\Om_{\rnd}$.
Let $\bar s_{\rnd}$ be the number of unique randomly sampled row indices.
As discussed in \cref{sec:comb-repe-rows}, let $\bar \xi_j$ be the $j$th unique row index.
Then we can define $\Om!_{\rnd} \in \Real^{\bar s_{\rnd} \times N}$ as follows:
\begin{equation}%
  \bar \omega(j,i) =
  \begin{cases}
    \sqrt{ \frac{c_j}{s} \frac{1-\pdet}{p_i}} & \text{if } \bar \xi_j = i \\
    0 & \text{otherwise}    
  \end{cases}
  \qtext{for all}
  (j,i) \in [\bar s_{\rnd}] \times [N].
\end{equation}

\subsection{The nonviable alternative of sampling without replacement}
\label{sec:why-not-use}

It might seem that an alternative to the hybrid deterministic and random
sampling method is to just increase the number of samples until
the number of unique rows is sufficiently large.
In other words, why not sample \emph{without} replacement?
The problem with this approach is that the samples are no longer independent,
and determining the appropriate weighting of the samples without the
independence assumption is infeasible.
Without the correct weighting, the sampled residual is no longer
an \emph{unbiased} estimator of the true residual.
The hybrid method is the closest we can get to this general idea of sampling
without replacement.
We include high-probability rows deterministically
so that the remainder are low enough probability that there
are relatively few repeats.
In fact, once we compute $\pdet$, we can estimate how many samples will be rejected,
which gives an estimate of the total number of samples (accepted and rejected) needed to achieve roughly the desired number of accepted samples.

\section{Efficient leverage score sampling for KRP matrices}
\label{sec:effic-matr-sketch}

Our aim is to use sketching for least squares where the matrix is 
a KRP matrix of the form $\Z = \Ak{d} \odot \cdots \odot \Ak{1} \in \Real^{N \times r}$ as defined in \cref{sec:backgr-least-squar}.
We cannot afford to explicitly form $\Z$ or explicitly compute its leverage scores since either would cost $O(N r^2)$.
In \cref{sec:leverage-score-upper}, we review how to upper bound
the leverage scores and use that to compute sampling probabilities so that $\beta$ in \cref{thm:sketching}
is $\beta = 1/r^{d-1}$.
Our main result in \cref{thm:krp-beta} shows that the number of samples needed for sampling KRP matrices 
is $s=O(r^d/ \epsilon)$.
In \cref{sec:implicit-random-row}, we describe how to sample rows according to the
probabilities established in the previous section without forming the probabilities or $\Z$ explicitly.
In \cref{sec:impl-determ-row}, we describe how to do the deterministic inclusion
proposed in \cref{sec:comb-determ-rand} without explicitly computing
all the probabilities.

\subsection{Sampling probabilities for KRP matrices and main theorem}
\label{sec:leverage-score-upper}

It is possible to obtain an upper bound on the leverage scores for $\Z$ by using the leverage scores for the factor matrices, as follows.
The result appears as \cite[Thm.~3.3]{ChPePeLi16} and \cite[Lemma~4]{BaBaKo18}.

\begin{lemma}[Leverage Score Bounds for KRP \cite{ChPePeLi16,BaBaKo18}] \label{lem:lev_bound}
  Let $\Z = \Ak{d} \odot \cdots \odot \Ak{1} \in \Real^{N \times r}$.
    Letting $(i_1,\dots,i_d)$ be the multi-index corresponding to $i$
as defined in \cref{eq:bijection}, the leverage scores can be bounded as
$\levZ \leq \levZ* \equiv \prod_{k=1}^d \levAk$.
\end{lemma}

Our main result for sketching the tensor least squares problem \cref{eq:lsq_krp} follows.
For sufficiently small $\epsilon$, 
$s=O(r^d/\epsilon)$ samples yields an $\epsilon$-accurate
residual with high probability.
The dependence on $1/\epsilon$ is similar to what is achieved
by CountSketch in the matrix case ($d=2$) \cite[Theorem~23]{Wo14}.

\begin{theorem}[Tensor Least Squares Sketching with Leverage Scores]\label{thm:krp-beta}
  Let $\Z = \Ak{d} \odot \cdots \odot \Ak{1} \in \Real^{N \times r}$ with $r < N$ and $\rank(\Z)=r$,
  $\Xmat \in \Real^{n \times N}$,
  and
  $\B* \equiv \arg \min_{\B \in \Real^{r \times n}} \| \Z \B' - \Xmat'\|^2$.
  Let $\Vp \in [0,1]^N$ be defined as
  \begin{equation}\label{eq:pi}
    p_i = \frac{\levZ*}{r^d}
    \qtext{where}
    \levZ* \equiv \prod_{k=1}^d \levAk
    \qtext{for all} i \in [N].
  \end{equation}
  For any  $\epsilon, \delta \in (0,1)$,
  set 
  $s=r^d \max \set{ C \log (r/\delta), 1/(\delta \epsilon) } $
  where $C=144/(1-1/\sqrt{2})^2$
  and let $\Om = \RandSample(s,\Vp)$.
  Then $\B~* \equiv \arg \min_{\B \in \Real^{r \times n}} \| \Om \Z \B' - \Om \Xmat\|_F^2$
  satisfies
  \begin{displaymath}
    \| \Z \B'~* - \Xmat' \|_F^2
    \leq (1+\epsilon) \| \Z \B'* - \Xmat' \|_F^2
  \end{displaymath}
  with probability at least $1-\delta$.
\end{theorem}

\begin{proof}
  From \cref{eq:pi} and \cref{lem:lev_bound}, we have
  \begin{displaymath}
    \frac{p_i r}{\ell_i(\Z)}
    = \frac{(\levZ*/r^{d}) \, r}{\levZ} \geq \frac{1}{r^{d-1}}
    \qtext{for all} i \in [N].
  \end{displaymath}
  Thus, setting $\beta = {1}/{r^{d-1}}$ satisfies the condition $\beta \leq \min_{i \in [N]} {p_i r}/{\ell_i(\Z)}$ in \cref{thm:sketching} and yields our result.
\end{proof}

An advantage of leverage-score sampling is bounding $\beta$ so that
we can get theoretical bounds,
but it also outperforms uniform sampling and two-norm-based sampling
in experiments as we show in \appsampling.

\subsection{Implicit random row sampling for  KRP matrices}
\label{sec:implicit-random-row}

Calculating the leverage scores for factor matrix $\Ak$ is inexpensive, costing $O(r^2n_k)$;
however, computing the sampling probabilities in \cref{eq:pi}
requires the Kronecker product of the leverages scores at a cost of $O(N)$.
To avoid this $O(N)$ expense, we sample from the distribution implicitly by
sampling each mode independently,
which is equivalent per the following result.

\begin{lemma}
\label{lemma:KRP_sample}
  Let $\Ak \in \Real^{n_k \times r}$ for $r \in [d]$, and let $\V{\ell}(\Ak)$ be the vector of leverage scores for $\Ak$.
  Let
  $i_k \sim \Multi(\V{\ell}(\Ak) / r)$
  for $k \in [d]$.
  The probability of selecting the multi-index $(i_1,\dots,i_d)$ is equal to 
  \begin{displaymath}
    p_i = \frac{\levZ*}{r^d}
    \qtext{where}
    \levZ* \equiv \prod_{k=1}^d \levAk
  \end{displaymath}
  and $i \in [N]$ is the linear index corresponding to $(i_1,\dots,i_d)$ with $N \equiv \prod_{k=1}^d n_k$.
\end{lemma}

Row $i$ of $\Z$ can be assembled in $O(rd)$ work by taking the Hadamard product of the rows of the factor matrices specified by the multi-index.

\subsection{Implicit computation of high-probability rows for deterministic inclusion}
\label{sec:impl-determ-row}

As explained in \cref{sec:comb-determ-rand}, it can be useful to deterministically include
all rows whose sampling probability is above a specified threshold, $\tau$.
However, it would be prohibitively expensive to find those above the threshold by explicitly computing all $N$ probabilities.
Instead, we perform a coarse-grained elimination of most candidate rows and then only compute the probabilities on a small subset of all rows.
Note that while this method can be run for arbitrary values of $\tau$, we aim to chose $\tau$ such that $\sdet$ is $O(1)$.

For each factor matrix, define the normalized leverage scores $\pk = \V{\ell}(\Ak)/r$
where the $i_k$th entry is denoted as $\pke$.
Recall that the probability of sampling row $\Z(i,:)$ is given by
\begin{displaymath}
  p_i = \frac{\prod_{k=1}^d \levAk}{r^d} = \prod_{k=1}^d \pke
    \qtext{for all} i \in [N],
\end{displaymath}
where $i$ is the linear index associated with subindices $(i_1,\dots,i_d)$.
The key insight is that only a subset of rows in each $\Ak$ could
possibly contribute to a row of $\Z$ with a sampling probability greater than $\tau$.

Our goal is to identify the set $\DetSet = \set{i \in [N] | p_i > \tau}$.
Define
\begin{displaymath}
  \alpha_k = \max_{i_k \in [n_k]} \pke, \;
  \alpha_* = \prod_{k=1}^d \alpha_k = \max_{i \in [N]} p_i ,
  \text{ and }
  \DetSetk = \set{i_k \in [n_k] | \pke > \tau \alpha_k / \alpha_*}.
\end{displaymath}
It is easy to show that if $i_k \not\in \DetSetk$, then $p_i \leq \tau$ for any linear index $i$ with row $i_k$ in its constituent subindices.
Hence, we can conclude
\begin{displaymath}
  \DetSet \subseteq \DetSetk{1} \otimes \cdots \otimes \DetSetk{d}.
\end{displaymath}
This means we need only check a small number of combinations.
If $\bar n_k = | \DetSetk |$, then we need only check $\prod_{k=1}^d \bar n_k \ll N$ possibilities.
It is easy to see that $\bar n_k < 1/\tau$, so we can limit the number of possibilities
to consider by the choice of $\tau$. We have found that $\tau = 1/s$ is effective in practice.
Technically, this only bounds the number of combinations that must be checked by $s^d$.
However, if such a problem is encountered, one can increase $\tau$ (which is chosen arbitrarily) so that the
number of combinations to check is reduced. In practice, we did not encounter a situation
of needing to check $s^d$ samples, so we used $\tau=1/s$ in all least squares solves of all of experiments.
More generally, we assume $\tau$ is sufficient small so that $\sdet < s$.

Once we have obtained the deterministic indices, we need to sample the remaining rows randomly
as described in \cref{sec:comb-determ-rand} for hybrid sampling.
Because we will not have explicit access to every sample's probability to rescale, 
we use \emph{rejection sampling}.
Suppose $\xi_j$ be the $j$th random sample, sampled according to the original probabilities in $\Vp$.
We reject the random sample $\xi_j$ if $\xi_j \in \DetSet$ and resample until $\xi_j \not \in \DetSet$.
This yields that the probability of selection $\xi_j = p_i / (1-\pdet)$, as desired.
We continue to sample in this manner until we have $s_{\rnd} = s - s_{\det}$ successful random samples.

The problem of determining which combination of independent random variables has a probability greater than a set threshold has also been studied in the context of cryptography.  Here a key is defined to be a string of the form $X_1, \ldots, X_m$ where each entry $X_i$ can take on one of $n$ values.  We are also supplied with a matrix of probabilities $\Mx{M}$ such that $M_{ij} = \Prob[X_i = j]$ and the assignment of each $X_i$.  The problem of key rank asks for a given probability $p$, how many keys have probability greater than $p$ and has been addressed in \cite{Glowacz15} and \cite{Martin15}.  If the key assignments with probability greater than $p$ are also to be returned, the problem is called key enumeration which has been addressed in \cite{Poussier16}.  Thus, key rank is similar to our problem of calculating $\sdet$ and key enumeration to finding the members of $\DetSet$.

\section{Alternating randomized least squares with leverage score sampling}
\label{sec:altern-rand-least}

In this section, we explain how all the algorithm components come together.
The sampling procedure to find the indices and weights (i.e., $\Om$) to
construct the reduced system is detailed in \cref{sec:sampl-krp-algor}.
Note that we avoid forming $\Om$ explicitly. Instead, we form $\Z~ \equiv \Om\Z$ and $\Xmat'~ \equiv \Om\Xmat'$ directly. %
The full CP algorithm that cycles through all modes of the tensors
and uses randomized sampling with the leverage scores is given in \cref{sec:full-algorithm}.
The computations are extremely efficient, and memory movement to extract the right-hand-side from the large tensor
$\X$ actually dominates cost in practice. We explain our method
for reducing those costs in \cref{sec:effic-sampl-from}.
Finally, the fit calculation is generally too expensive to compute
exactly for tensors with billions of nonzeros, so
we estimate the fit as described in \cref{sec:estimating-fit}.

\subsection{Finding indices and weights}
\label{sec:sampl-krp-algor}

The first and most important step is identifying rows and associated weights for inclusion in the random reduced subproblem.
\Cref{alg:skrplev} outlines procedure \textsc{SkrpLev} for finding these.
The inputs are the normalized leverage scores for each factor matrix ($\pk = \V{\ell}(\Ak)/r$ for $k=1,\dots,d$),
the number of samples ($s$), and the deterministic threshold ($\tau$).
A few notes are in order.
\begin{algorithm}[t]
  \caption{Hybrid Deterministic and Random  KRP Indices by Leverage Score}
  \label{alg:skrplev}
  \begin{algorithmic}[1]\small%
    \Function{\textsc{SkrpLev}}{$\pk{1},\dots,\pk{d}, s, \tau$}
    \Comment{$\pk \equiv \V{\ell}(\Ak)/r$}
    \State\label{stp:didx}%
    $(\idxdet, \sdet, \pdet) \gets \DIDX(\pk{1},\dots,\pk{d},\tau)$
    \Comment Find $\set{i \in [N] | p_i > \tau}$
    \State $\srnd \gets s - \sdet$
    \State\label{stp:sidx}%
    $(\idxrnd, \wgtrnd) \gets \SIDX(\pk{1},\dots,\pk{d},\srnd,\tau,\pdet)$
    \Comment See \cref{alg:skrpidx}
    \State\label{stp:cidx}%
    $(\idxrnd, \wgtrnd, \srnd*) \gets \CIDX(\idxrnd, \wgtrnd)$
    \Comment{\emph{After} rejection sampling}
    \State $\idx \gets \textsc{cat}(\idxdet, \idxrnd)$
    \State $\wgt \gets \textsc{cat}(\mathbf{1}_{\sdet}, \wgtrnd)$    
    \Comment Weights for deterministic indices is 1
    \State $s \gets \sdet + \srnd*$
    \State \Return{$(\idx, \wgt, s)$}
    \Comment{Return indices and weights}
    \EndFunction
  \end{algorithmic}
\end{algorithm}
\begin{itemize}[leftmargin=0.5cm]
\item Function $\DIDX$, called in \cref{stp:didx}, computes hybrid deterministic indices
  \begin{displaymath}
    \idxdet = \Set{i \in [N] | p_i = \prod_{k=1}^d \pke > \tau},
    \;
    \sdet = | \idxdet |,
    \text{ and }
    \pdet = \sum_{i \in \idxdet} p_i.
  \end{displaymath}
  without explicitly computing all the probabilities, per \cref{sec:impl-determ-row}.
  We assume that $\sdet < s$. If not, we take the $s$ highest probabilities that are found.
  Setting $\tau = 1$ means that no samples are included deterministically ($\idxdet = \emptyset, \sdet=0, \pdet=0$).
\item Function $\SIDX$, called in \cref{stp:sidx}, is detailed in \cref{alg:skrpidx}.
  It randomly samples indices $i \sim \Multi(\Vp)$ where $p_i = \prod_{k=1}^d \pke$ for all $i\in[N]$.
  Any sample with $p_i > \tau$ is rejected, and  the weights are correspondingly adjusted by
  multiplying them by $\sqrt{1-\pdet}$.
  The same index may be sampled multiple times.
  Our actual implementation samples and rejects indices in bulk, oversampling
  to ensure that we still have at least $\srnd$ indices after the rejection is complete.
\item Function $\CIDX$ combines multiple indices as described in \cref{sec:comb-repe-rows}.
  If row $i$ appeared $c_i$ times in $\idxrnd$, then its weight is scaled by $\sqrt{c_i}$.
  The count $\srnd*$ is the number of \emph{unique} indices in that was produced by $\SIDX$.
\end{itemize}

\begin{algorithm}[t]
  \caption{Random KRP Indices by Leverage Score}
  \label{alg:skrpidx}
  \begin{algorithmic}[1]\small%
    \Function{\SIDX}{$\pk{1},\dots,\pk{d}, \srnd, \tau, \pdet$}
    \Comment{$\pdet \equiv \sum_{p_i > \tau} p_i$}
    \State{$j \gets 0$} 
    \While{$j<\srnd$}
    \For{$k=1,\dots,d$}
    \Comment{Sample random index $i \equiv (i_1,\dots,i_k) \in [N]$}
    \State{$i_k \gets \Multi(\pk)$}
    \EndFor
    \State $p_i \gets \prod_{k=1}^d \pke$
    \If{$p_i \leq \tau$} \Comment{Reject if $p_i > \tau$}
    \State{$\idxrnd(j) \gets i$} 
    \State{$\wgtrnd(j) \gets \sqrt{ \frac{1-\pdet}{p_i \, \srnd }  }$}
    \Comment{Weight adjusted for rejected indices}
    \State{$j \gets j+1$}
    \EndIf
    \EndWhile
    \State \Return{$(\idxrnd, \wgtrnd)$}
    \EndFunction
  \end{algorithmic}
\end{algorithm}

\subsection{Full algorithm}
\label{sec:full-algorithm}

The CP tensor decomposition of rank $r$ for an order-$(d+1)$ tensor
is defined by $(d+1)$ factor matrices $\Ak{1},\dots,\Ak{d+1}$ that minimizes
the sum of the squares error between the data tensor $\X$ and CP model $\M$.
We use the shorthand $\M = \KT$ where
$m(i_1,\dots,i_{d+1}) = \sum_{j=1}^r \prod_{k=1}^d a_{k}(i_k,j)$.
The standard \cpals algorithm solves for each factor matrix in turn (inner iterations), keeping the others fixed.
Each least squares problem is of the form shown in \cref{eq:lsq_krp}. Although \cref{eq:lsq_krp}
is specific to solving for $\Ak{d+1}$, this is really just a notational convenience.
Each outer iteration, 
we compute the proportion of the data described by the model, i.e., 
\begin{equation}
  \label{eq:fit}
  \text{fit} = 1 - \frac{ \| \X - \M \|}{\| \X \|}.
\end{equation}
The method halts when the fit ceases to improve by at least $10^{-4}$.
We refer the reader to \cite{KoBa09} for further details and references on \cpals.

Our randomized variant \cprandlev is presented in \cref{alg:cp-arls-lev}.
The inputs are the order-$(d+1)$ tensor, $\X$; the desired rank, $r \in \mathbb{N}$;
the number of samples for each least squares problem, $s \in \mathbb{N}$;
the deterministic cutoff, $\tau \in [0,1]$ (defaults to 1 for random and $1/s$ for hybrid);
the number of outer iterations per epoch, $\eta \in \mathbb{N}$ (which defaults to 5);
the number of failed epochs allowed before convergence, $\pi$ (which defaults to 3);
and the initial guesses for the factor matrices.
We group the iterations into epochs of $\eta$ outer iterations
since the method does not necessarily improve with every step due to the randomness.
Further, we may not want to quit until the fit fails to improve for $\pi$ epochs.
In many cases, computing the fit exactly would be to expensive,
so we use the approximate fit as documented in \cref{sec:estimating-fit}.

We presented the canonical least squares problem in \cref{eq:lsq_krp} in terms of the specific least squares problem for mode $d+1$, but
the \cpals method requires that we solve such a problem for every mode. This is an important implementation detail
but does not otherwise require any change in thinking.
At inner iteration $k$, for instance, we can call the \textsc{SkrpLev} methods with
$d$ leverage scores vectors --- the only difference is that we leave out the $k$th vector
of leverage scores rather than the $(d+1)$st.
We let $\bar s$ denote the actual number of sampled rows, which may be less than $s$ due to
combining repeated rows.
The function \textsc{KrpSamp} builds the sampled KRP matrix given the factor matrices, indices of the rows to be combined,
and corresponding weights. The work to construct $\Z~$ is $O(\bar s d r)$.
The function \textsc{TnsrSamp} extracts the appropriate rows of the unfolded matrix as described in \cref{sec:effic-sampl-from}.
(Note that we do not explicitly provide algorithms for \textsc{KrpSamp} or \textsc{TnsrSamp}.)
The solution of the least squares problem costs $O(\bar s r^2 + \bar s n r)$, where again the $O(\bar s n r)$ complexity can be reduced by exploiting the sparsity of the sampled tensor.
We use the same $s$ for every mode of the tensor, and making $s$ mode dependent is a topic for future work; see, e.g., \cite{AgGiYe20}.
The end-to-end complexity is presented in \appcomplexity.

\begin{algorithm}[t]
  \caption{CP via Alternating Randomized Least Squares with Leverage Scores}
  \label{alg:cp-arls-lev}
  \begin{algorithmic}[1]\small
    \Function{\cprandlev}{$\X$, $r$, $s$, $\tau$, $\eta$, $\pi$, $\texttt{tol}$, $\set{\Ak}$}
    \For{$k=1,\dots,d+1$}
    \State $\pk \gets \V{\ell}(\Ak)/r$ \Comment{Compute scaled leverage scores for initial guess}
    \EndFor
    \Repeat
    \For{$\ell=1,\dots,\eta$} \Comment{Group outer iterations into epochs}
    \For{$k=1,\dots,d+1$} \Comment{Cycle through tensor modes}
    \State $(\idx,\wgt,\bar s) \gets \textsc{SkrpLev}(\pk{1},\dots,\pk{k-1},\pk{k+1},\dots,\pk{d+1}, s, \tau)$
    \Comment{$\bar s \leq s$}
    \State $\Z~ \gets \textsc{KrpSamp}(\Ak{1},\dots,\Ak{k-1},\Ak{k+1},\dots,\Ak{d+1},\idx,\wgt)$
    \Comment $\Z~ \in \Real^{\bar s \times r}$
    \State $\Xmat~ \gets \textsc{TnsrSamp}(\X,k,\idx,\wgt)$
    \Comment $\Xmat~ \in \Real^{\bar s \times n_k}$
    \State $\Ak \gets \argmin_{\B \in \Real^{n_k \times r}} \| \Z~ \B' - \Xmat'~ \|$
    \State $\pk \gets \V{\ell}(\Ak)/r$
    \EndFor
    \EndFor
    \State Compute \texttt{fit} (exact or approximate) \Comment{Computed only after each epoch}
    \Until{\texttt{fit} has not improved by more than $\texttt{tol}$ for $\pi$ subsequent epochs}
    \State \Return{$\KT$}
    \EndFunction
  \end{algorithmic}
\end{algorithm}

\subsection{Efficient sampling from sparse tensor}
\label{sec:effic-sampl-from}

A final consideration for efficiency is quickly compiling the right hand side, $\Xmat~$.
Recall that $\Xmat$ is the $(d+1)$-mode unfolding of the $(d+1)$-way tensor $\X$.
The tensor $\X$ is sparse, so we store only the nonzeros. We use coordinate format which stores the
coordinates $(i_1, \dots, i_{d+1})$ and value $x_{i_1\dots,i_{d+1}}$ for each nonzero \cite{BaKo07},
for a total storage of $(d+2) \nnz(\X)$ for a $(d+1)$-way tensor $\X$.

The mode-$(d+1)$ unfolding of $\X$ produces a matrix of size $n \times N$ where $N = \prod_{k=1}^d n_k$ and $n = n_{d+1}$.
We need to select and reweight the $s$ columns of $\Xmat$ (rows of $\Xmat'$) that correspond to the selected rows of $\Z$, per \cref{alg:skrplev}.
In this way, we can quickly build the sparse matrix $\Xmat~$ as follows:
\begin{displaymath}
  \Xmat~(i,j) =
  \begin{cases}
    \wgt(j) \, x(i_1,\dots,i_d,i_{d+1}) & \text{if } \idx(j) = (i_1,\dots,i_d) \text{ and } i = i_{d+1} \\
    0 & \text{otherwise}.
  \end{cases}
\end{displaymath}
We typically store the entries of $\idx$ as linear indices, so, 
for efficiency, we recommend to precompute and store the linearized indices corresponding to $(i_1,\dots,i_d)$.
Further, we operate on all $(d+1)$ modes, so these should be precomputed for every mode. This
requires $(d+1) \nnz(\X)$ additional storage.

\subsection{Estimating the fit}
\label{sec:estimating-fit}

The tensor fit \cref{eq:fit} is used in the stopping condition for \cprandlev.
Unfortunately, calculating the fit costs $O(r\nnz(\X))$ and can therefore take many times longer than an epoch.
Instead, we \emph{estimate} the fit based on a random sample of tensor elements as in \cite{BaBaKo18}.
Since we are focused on sparse tensors, we have the additional difficulty that a uniform sample returns predominantly zero entries, by the definition of sparsity, and is likely to lead to an inaccurate estimate of the fit.  
To correct this, we use a technique called stratified sampling to sample a specified proportion of zero and nonzero entries \cite{KoHo20}.
Let $s_{\text{fit}}$ be the user-specified number of samples
and $\alpha \in [0,1]$ be the fraction of nonzero elements (by default we use $\alpha = 0.5$).  
We sample $\lceil \alpha\sfit \rceil$ nonzero entries  uniformly at random and $\lfloor (1-\alpha)\sfit \rfloor$ zero entries elements uniformly at random.
The zeros are selected via \emph{rejection sampling} as described in \cite{KoHo20}.
The result is a set of $s_{\text{fit}}$ linear indices, denoted by $\mathcal{F}$.
If we define $F = \| \X - \M \|^2$, then the fit \cref{eq:fit} is $1 - \sqrt{F}/\|\X\|$. We estimate $F$ as
\begin{equation}
  \widehat{F} = \sum_{i \in \mathcal{F}} \phi_i (m_i - x_i)^2
  \qtext{where}
  \phi_i =
  \begin{cases}
    \nnz(\X) / \lceil \alpha\sfit \rceil & \text{if } x_i \neq 0, \\
    (n^{d+1} - \nnz(\X)) / \lfloor (1-\alpha)\sfit \rfloor & \text{if } x_i = 0,
  \end{cases}
\end{equation}
where $m_i$ and $x_i$ are the $i$th entries of $\M$ and $\X$, respectively. The $\M$ is updated at each step of the algorithm.
For the estimated fit, we use $1 - \sqrt{\widehat{F}}/\|\X\|$.
We sample the elements of the tensor once at the beginning of the algorithm and then use that sample for all subsequent estimates.

\section{Numerical results}
\label{sec:numerical-results}
We present experiments to demonstrate the
benefits of combining repeated rows in sketching algorithms (\cref{sec:exp-comb-repeat}),
differences between random and hybrid sampling (\cref{sec:hybrid-rand}), 
comparisons between our proposed and standard methods
(\cref{sec:uber,sec:amazon,sec:reddit}), and example factors for a massive tensor (\cref{sec:reddit}).
The methods we compare are
\begin{itemize}[leftmargin=0.5cm]
\item \textbf{Random}: Proposed CP-ARLS-LEV method (\cref{alg:cp-arls-lev}) using $\tau=1$ (i.e., no deterministic inclusion of rows), $\eta = 5$, $\pi = 3$, and $\texttt{tol} = 10^{-4}$
\item \textbf{Hybrid}: Proposed CP-ARLS-LEV method, same as above except $\tau=1/s$ (i.e., high-probability rows included deterministically)
\item \textbf{Standard}: CP-ALS method with standard settings, including $\texttt{tol}=10^{-4}$
\item \textbf{SPALS}: Our implementation of the method proposed in \cite{ChPePeLi16}. We mirror the settings for CP-ARLS-LEV as much as possible, using $\eta=5$, $\pi = 3$, and $\texttt{tol} = 10^{-4}$. 
\end{itemize}

\begin{table}
  \centering\footnotesize
  \begin{tabular}{|l|l|l|l|}
    \hline
    \bf Name & \bf Size & \bf Nonzeros & \bf Density \\ \hline
    Uber &  183 $\times$ 24 $\times$ 1,140 $\times$ 1,717 &  3,309,490 & 0.038\% \\ %
    Amazon & 4,821,207 $\times$ 1,774,269 $\times$ 1,805,187 & 1,741,809,018 & 1.1 $\times 10^{-8}$\% \\ %
    Reddit & 8,211,298 $\times$ 176,962 $\times$ 8,116,559 & 4,687,474,081 & 4.0 $\times 10^{-8}$\% \\ \hline
  \end{tabular}
  \caption{Large-scale tensors used in experiments}
  \label{tab:tensors}
\end{table}

\noindent
The experiments are based on three large-scale sparse tensors from FROSTT \cite{frosttdataset}, the largest having nearly 5B nonzeros,
whose sizes are given in \cref{tab:tensors}. Briefly, we describe the tensors below. 
\begin{itemize}[leftmargin=0.5cm]
\item \textbf{Uber}:  4-way count tensor of New York City Uber pickups in April--August 2014,  183 days ${\times}$ 24 hours ${\times}$ 1.1K latitudes ${\times}$ 1.7K longitudes with 3M~nonzeros.
\item \textbf{Amazon}: 3-way count tensor based on user review text of 5M users $\times$ 2M words $\times$ 2M reviews with 2B nonzeros.  
\item \textbf{Reddit}: 3-way log-count\footnote{This tensor has been modified from the raw count tensor provided by FROSTT. Each entry is $\log(c+1)$ where $c$ is the count. Note that the zeros are unchanged. This is a standard weighting in text analysis. The primary effect is that the largest entries are damped.} tensor of social network postings of 8M users $\times$ 200K subreddits $\times$ 8M words with 5B nonzeros.
\end{itemize}

\noindent
Factor matrices are initialized by drawing each entry from a standard Gaussian distribution.
We explore the randomized range finder initialization in \apprrf.

All experiments were run using MATLAB (Version 2018a) using the Tensor Toolbox for MATLAB \cite{TTB_Software}.
We used three computational environments: (A) a Dual Socket Intel E5-2683v3 2.00~GHz CPU (28 total cores) with 256~GB memory, (B) a Dual Socket AMD Epyc 7601 2.20~GHz CPU (64 total cores) with 1~TB memory, and (C) a Dual Socket AMD EPYC 7452 (64 total cores) with 256~GB memory.

As mentioned in some earlier sections, we provide additional results in the appendices and supplements.
Comparisons of leverage-score-based sampling to uniform and length-squared-based sampling is provided in \appsampling.
Comparisons to CP-ARLS for dense
tensors  \cite{BaBaKo18} is provided in \appdense.

\subsection{Combining repeated rows}
\label{sec:exp-comb-repeat}

We show that combining repeated rows (\cref{sec:comb-repe-rows}) can significantly decrease the time dedicated to solving the sampled system.
This does not change the sampling methodology nor the solution; instead, it merely introduces a preprocessing step that removes redundancies in the linear system. 
We rarely see repeated rows with random factor matrices since they are incoherent,
but we often see repeated rows as we converge toward a solution because the leverage scores are skewed.

\pgfplotstabletranspose[input colnames to={IDs}, colnames from={IDs}]{\amazonrepeat}{data-amazon-repeated-rows.csv}
\pgfplotstabletranspose[input colnames to={IDs}, colnames from={IDs}]{\uberrepeat}{data-uber-repeated-rows.csv}
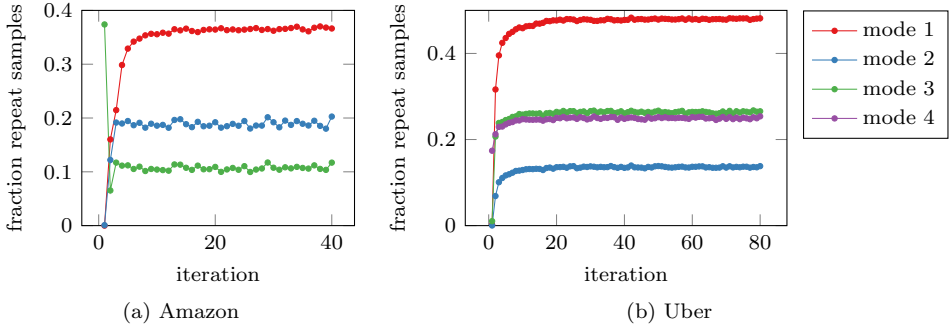
\begin{figure}
  \centering
  \pgfplotsset{
    ymin = 0,
  }
  \subfloat[Amazon]{
    \begin{tikzpicture}[baseline]
      \begin{axis}[
        height=1.75in,width=0.4\textwidth,
        font=\footnotesize,   
        xlabel = iteration,
        ylabel = fraction repeat samples,
        ymax = 0.4,
        unbounded coords=jump,
        mark size=1pt,
        ]
        \addplot+[mark=*,mark size=1pt] table[y=M1_R]{\amazonrepeat};
        \addplot+[mark=*] table[y=M2_R]{\amazonrepeat};
        \addplot+[mark=*] table[y=M3_R]{\amazonrepeat};
      \end{axis}
    \end{tikzpicture}~~}%
  \subfloat[Uber]{
    \begin{tikzpicture}[baseline]
      \begin{axis}[
        height=1.75in,width=0.45\textwidth,
        font=\footnotesize,   
        xlabel = iteration,
        ylabel = fraction repeat samples,
        ymax = 0.5,
        unbounded coords=jump,
        legend style={at={(1.05,1)},anchor=north west},
        mark=*,
        mark size=1pt,
        ]
        \addplot+[mark=*] table[y=M1_R]{\uberrepeat};
        \addplot+[mark=*] table[y=M2_R]{\uberrepeat};
        \addplot+[mark=*] table[y=M3_R]{\uberrepeat};
        \addplot+[mark=*] table[y=M4_R]{\uberrepeat};
        \legend{mode 1,mode 2, mode 3, mode 4}
      \end{axis}
    \end{tikzpicture}~~}%
  \caption{Fraction of repeated row samples, i.e., $(s -\bar s)/s$, at each iteration and for each mode of a single run of CP-ARLS-Lev (Random) on two tensors
    using $s = 2^{17}$ rows.}
  \label{fig:repeated-rows}
\end{figure} 

\Cref{fig:repeated-rows} demonstrates the prevalence of repeated rows as the factor matrices converge towards a solution.
We run CP-ARLS-Lev (Random) once each for Amazon and  Uber, recording the proportion of repeats for each mode in each iteration.
The fraction repeats becomes substantial after only a few iterations, motivating
the combination of repeat rows.

\pgfplotstabletranspose[input colnames to={IDs}, colnames from={IDs}]{\amazoncombine}{data-amazon-combinetime.csv}
\pgfplotstabletranspose[input colnames to={IDs}, colnames from={IDs}]{\amazoncombine}{data-amazon-combinetime.csv}
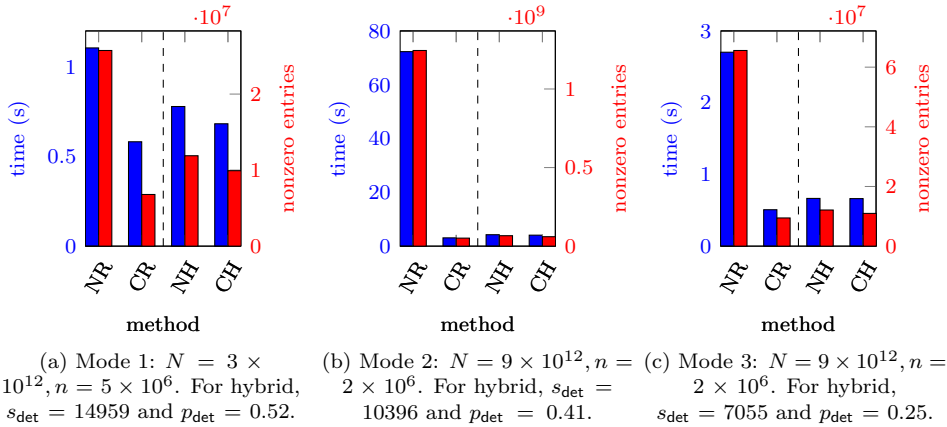
\begin{figure}
  \centering
  \pgfplotsset{
    xticklabels = {NR,CR, NH, CH},
    boxplot/draw direction = y,
    height=1.75in,width=0.28\textwidth,
    ymin = 0,
    xtick = {1, 2, 3, 4},
    xticklabel style = {align=center, font=\footnotesize, rotate=60},
    xlabel = method,
  }
  \subfloat[Mode 1: $N=3 \times 10^{12}, n=5 \times 10^6$. For hybrid, $\sdet = 14959$ and $\pdet = 0.52$.]{
  \begin{tikzpicture}[baseline]
    \begin{axis}[
    font=\footnotesize,
    bar width = 5 pt,
    bar shift = - 2.5 pt,
    every y tick label/.append style={blue}, 
    ytick pos = left, ymax=1.2,
      ylabel = {\textcolor{blue}{time (s)}},
      ]      
      \addplot[ybar,fill=blue] table[y=M1] {\amazoncombine};
      \addplot[color=black, dashed] coordinates {(2.5,0) (2.5,1.2)};
    \end{axis}
    \begin{axis}[
      font=\footnotesize,
      bar width = 5 pt,
      bar shift = 2.5 pt,
      every y tick label/.append style={red}, 
      axis y line* = right,
      ytick pos = right,
      ylabel = {\textcolor{red}{nonzero entries}},
      ]
      \addplot[ybar,fill=red] table[y=M4] {\amazoncombine};
    \end{axis}
  \end{tikzpicture}
  }
  \subfloat[Mode 2: $N=9 \times 10^{12}, n=2 \times 10^6$. For hybrid, $\sdet = 10396$ and $\pdet = 0.41$.]{
    \begin{tikzpicture}[baseline]
    \begin{axis}[
      font=\footnotesize,
      bar width = 5 pt,
      bar shift = - 2.5 pt,
      every y tick label/.append style={blue}, 
      ytick pos = left, ymax=80,
      ylabel = {\textcolor{blue}{time (s)}},
      ]
      \addplot[ybar,fill=blue] table[y=M2] {\amazoncombine};
      \addplot[color=black, dashed] coordinates {(2.5,0) (2.5,80)};
    \end{axis}
    \begin{axis}[
    font=\footnotesize,
      bar width = 5 pt,
      bar shift = 2.5 pt,
      every y tick label/.append style={red}, 
      axis y line* = right,
      ylabel = {\textcolor{red}{nonzero entries}},
      ]
      \addplot[ybar,fill=red] table[y=M5] {\amazoncombine};
    \end{axis}
    \end{tikzpicture}
  }
  \subfloat[Mode 3: $N=9 \times 10^{12}, n=2 \times 10^6$. For hybrid, $\sdet =  7055$ and $\pdet = 0.25$.]{
    \begin{tikzpicture}[baseline]
    \begin{axis}[
    font=\footnotesize,
      bar width = 5 pt,
      bar shift = - 2.5 pt,
      every y tick label/.append style={blue}, 
      ytick pos = left,ymax=3,
      ylabel = {\textcolor{blue}{time (s)}},
      ]
      \addplot[ybar,fill=blue] table[y=M3] {\amazoncombine};
      \addplot[color=black, dashed] coordinates {(2.5,0) (2.5,3)};
    \end{axis}
    \begin{axis}[
      font=\footnotesize,
      bar width = 5 pt,
      bar shift = 2.5 pt,
      every y tick label/.append style={red}, 
      axis y line* = right,
      ylabel = {\textcolor{red}{nonzero entries}},
      ]
      \addplot[ybar,fill=red] table[y=M6] {\amazoncombine};
    \end{axis}
  \end{tikzpicture}
    }
    \caption{Comparing combing rows (C) with not combining (N) for both random (R) and hybrid (H)  sampling schemes for least squares sketching.
      We report both the combined reweight and solution time (blue) and number of nonzeroes in the sampled right-hand side (red), which closely correlate.
      The matrix is of size $N \times 10$, with $n$ right-hand sides, based on modes 1--3 of the Amazon tensor and factor matrices
      from a solution of rank $r=10$. The methods sample $s = 2^{17}$ rows, random uses $\tau=1$, and hybrid uses $\tau = 1/s$.
      The results are averaged across 10 runs.  Note that each mode is on a different scale.}
  \label{fig:amazon-combine}
\end{figure} %

To show the impact of combining rows for solving the least squares problem,
we use the \emph{solution} factor matrices for the Amazon tensor
from a run of \cpals with rank $r = 10$ (producing a final fit of 0.3055),
using computational environment (B).
\Cref{fig:amazon-combine} shows the results for problems of the form in \cref{eq:lsq_krp} with $d=2$, $r=10$,
and different values of $N$ and $n$, for
both the random and hybrid leverage score sampling schemes.
(Note that we have not provided sufficient information here to compare the random and hybrid schemes ---
comparing those two methods is explored in the next subsection.)
First, the time to solution, plotted in blue, is \emph{always} improved by combining rows.
The difference is particularly dramatic for the random method in mode 2.
Second, although all the systems are roughly the same size ($N=O(10^{12} )$ and $n=O(10^6)$),
the solution time depends on the number of nonzeros in the right-hand side, which is plotted in red.
Combing rows also reduces the nonzeros and correlates closely to the improvements in time.
Third, the reductions are much smaller for the hybrid sampling since deterministic inclusion of high-probability rows results in fewer repeats.

In general, the dominant cost for each subproblem is usually in the setup for the least squares solve.
Specifically, extracting the sampled fibers (discussed in \cref{sec:effic-sampl-from}) from the sparse tensor
costs approximately 14s for mode 1, 25s for mode 2, and 11s seconds for mode 3.
In our implementation, this step has small variance in cost between the four methods as we find the unique fiber indices before extracting the relevant fibers of the tensor.
But as mode 2 in our experiments show, without combining repeats or hybrid sampling the solve can eclipse the fiber extraction and greatly increase the overall runtime.

As there is essentially no downside to combing rows,
we recommend that combining rows should be a default in any leverage-score based sampling scheme.
Henceforth, CP-ARLS-LEV (random and hybrid) does so in all experiments.

\subsection{Comparing random and hybrid sampling}
\label{sec:hybrid-rand}

We consider a specific least squares problem to illustrate the differences between random and hybrid sampling.
We use the Uber tensor and solve for the first factor matrix on environment (A),
fixing the factors for modes 2--4 to a solution produced by \cpals.
This corresponds to a least squares problem of the form \cref{eq:lsq_krp} with $N =$ 46,977,120 rows,
$r=10$ columns, and $n=183$ right-hand sides.
The example was chosen for two reasons:
(1) It is small enough to compute the true solution, which would be too expensive for the larger tensors.
(2) It shows a marked difference between random and hybrid sampling because it has concentrated leverage score structure.

\pgfplotstabletranspose[input colnames to={IDs}, colnames from={IDs}]{\ubereps}{data-uber-epsilon.csv}
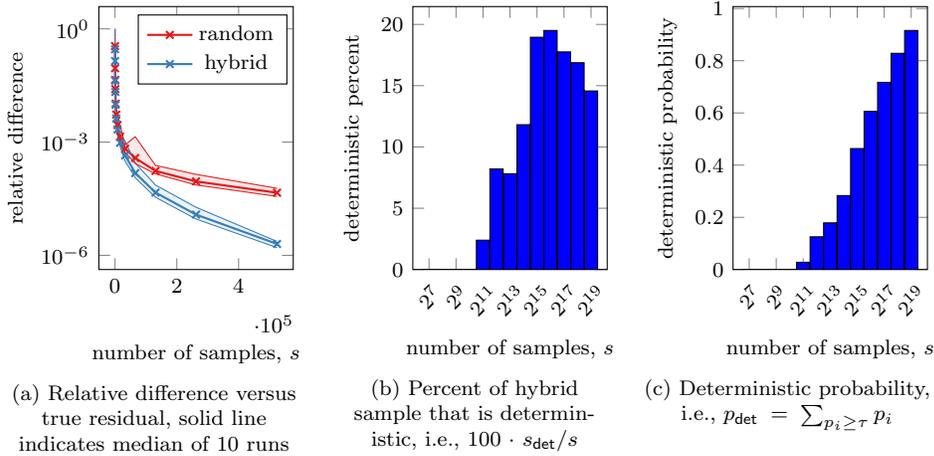
\begin{figure}[t]
  \centering
  \subfloat[Relative difference versus true residual, solid line indicates median of 10 runs]{\label{fig:uber-resids}
    \begin{tikzpicture}[baseline]
      \begin{axis}[
        height=2in,width=0.32\textwidth,
        font=\footnotesize,
        ymode = log,    
        xlabel = {number of samples, $s$},
        ylabel = {relative difference},
        x label style={at={(0.5,-0.25)}},
        unbounded coords=jump,
        legend pos = north east,
        cycle list={
          {index of colormap=0 of Set1-8,mark=*},
          {index of colormap=1 of Set1-8,mark=*}}
        ]
        \pgfplotsinvokeforeach {1,2} {
          \addplot+[thick, mark=x] table[x=S, y=M#1]{\ubereps};
        }
        \pgfplotsinvokeforeach {1,2} {
          \addplot+[name path={max#1},thin,mark=none] table[x=S, y=M#1_MAX]{\ubereps};
        }
        \pgfplotsinvokeforeach {1,2} {
          \addplot+[name path={min#1},thin,mark=none] table[x=S, y=M#1_MIN]{\ubereps};
        }
        \pgfplotsinvokeforeach {1,2} {
          \addplot+[.!20, fill opacity=0.5] fill between[of = max#1 and min#1];
        }
        \legend{random, hybrid}
      \end{axis}
    \end{tikzpicture}~}%
  \subfloat[Percent of hybrid sample that is deterministic, i.e., $100\cdot\sdet/s$]{\label{fig:uber-drows}~
    \begin{tikzpicture}[baseline]
      \begin{axis}[
          height=2 in,width=0.32\textwidth,
          font=\footnotesize,
          xlabel = {number of samples, $s$},
          ylabel = {deterministic percent},
          ymin = 0,
          /pgf/bar width=5pt,
          xtick = {1, 3, 5, 7, 9, 11, 13},
          xticklabel style = {align=center, font=\footnotesize, rotate=45},
          xticklabels = {$2^{7}$,$2^{9}$,$2^{11}$,$2^{13}$,$2^{15}$,$2^{17}$,$2^{19}$},
          ]
          \addplot[ybar,fill=blue] table[y=SDET] {\ubereps};
        \end{axis}   
      \end{tikzpicture}~}%
    \subfloat[Deterministic probability, \\ i.e., $\pdet = \sum_{p_i\geq \tau} p_i$]{\label{fig:uber-dprop}~
      \begin{tikzpicture}[baseline]
        \begin{axis}[
          height=2 in,width=0.32\textwidth,
          font=\footnotesize,
          xlabel = {number of samples, $s$},
          ylabel = {deterministic probability},
          ymin = 0,
          /pgf/bar width=5pt,
          xtick = {1, 3, 5, 7, 9, 11, 13},
          xticklabel style = {align=center, font=\footnotesize, rotate=45},
          xticklabels = {$2^{7}$,$2^{9}$,$2^{11}$,$2^{13}$,$2^{15}$,$2^{17}$,$2^{19}$},
          ]
          \addplot[ybar,fill=blue] table[y=PDET] {\ubereps};
        \end{axis}
      \end{tikzpicture}}
    \caption{Single least squares problem with $N =$ 46,977,120 rows, $r=10$ columns, and $n=183$ right-hand sides,
      corresponding to solving for the first factor matrix in the Uber problem ($d=3$).
      Random uses $\tau=1$ and hybrid uses $\tau=1/s$.
  }
  \label{fig:uber-epsilon}
\end{figure} %

\Cref{fig:uber-resids} shows the relative residual difference between the sampled solution
and the exact solution as the number of samples increases from $2^7$ to $2^{19}$.
Specifically, using the notation of \cref{thm:krp-beta}, the y-axis corresponds to
\begin{displaymath}
  \frac{ \left|  \| \Z \B'~* - \Xmat' \|_F^2 - \| \Z \B'* - \Xmat' \|_F^2 \right| } {\max \set{1, \| \Z \B'* - \Xmat' \|_F^2}}.
\end{displaymath}
We compare random sampling method and the hybrid-deterministic sampling with $\tau = 1/s$.
For each number of samples, we solve the least squares problem 10 times, and the error bars indicate the range of values
obtained while the solid line denotes the median.
Note that the maximum number of samples, $s=2^{19}$, represents only 1.1\% of the rows in the matrix and achieves an accuracy of $10^{-4}$.
For this problem, hybrid sampling clearly improves over random sampling, obtaining approximately 2 more digits of accuracy for $s=2^{19}$ samples.
\Cref{fig:uber-drows} show the fraction of the hybrid sample that is deterministically included ($\sdet/s$), which peaks at less than 20\%.
\Cref{fig:uber-dprop} shows the fraction of the
total sampling probability contained in these samples ($\pdet$), which goes as high as 90\%.
Note that these values are the same across all runs as they are deterministic based on the threshold $\tau$.
If we look at the case of $s=2^{19}$, this means that 15\% of the rows that are included deterministically
would account for approximately 90\% of the sampled rows in the random method.
Hence, we are getting better ``sample efficiency'' in terms of the number of unique samples
with the hybrid method, but the trade-off is added computational cost
in terms of computing $\pdet$ (can be somewhat expensive) and the rejection sampling (relatively inexpensive).
The next sections explore some tradeoffs between these approaches.

\subsection{Equivocal performance on small tensor}
\label{sec:uber}

We consider the rank $r=25$ CP tensor decomposition using matrix sketching on the Uber tensor.
This is a relatively small tensor, so we can compute its fit exactly.

\pgfplotstabletranspose[input colnames to={IDs}, colnames from={IDs}]{\uberraw}{data-uber-fittrace-raw.csv}
\pgfplotstabletranspose[input colnames to={IDs}, colnames from={IDs}]{\uberitp}{data-uber-fittrace-interpl.csv}
\pgfplotstabletranspose[input colnames to={IDs}, colnames from={IDs}]{\ubertime}{data-uber-totaltimes.csv}
\pgfplotstabletranspose[input colnames to={IDs}, colnames from={IDs}]{\uberfit}{data-uber-finalfits.csv}
\begin{figure}[t]
  \centering
  \subfloat[Box plot of final fit and total time over 10 runs.]%
  {\label{fig:uber-fit-time}%
    \begin{tikzpicture}
  \begin{axis}[
    boxplot/draw direction = x,
    y dir=reverse,
    height=1.75in,width=0.45\textwidth,
    font=\footnotesize,
    xlabel = fit,
    xticklabel style={%
       /pgf/number format/.cd,
           fixed,
           fixed zerofill,
           precision=3,
           },
    cycle list={[indices of colormap={4, 0, 5, 1, 6, 2, 3 of Set1-8}]},
    ytick = {1, 2, 3, 4, 5, 6, 7},
    yticklabel style = {align=center, font=\footnotesize}, %
    yticklabels = {Random $s=2^{15}$,Hybrid $s=2^{15}$,Random $s=2^{16}$, Hybrid $s=2^{16}$, Random $s=2^{17}$,Hybrid $s=2^{17}$, Standard},
    ]
    \foreach \m in {5, 1, 6, 2, 7, 3, 4} {
      \addplot+[boxplot, fill, draw=black] table[y=M\m] {\uberfit};
    }
  \end{axis}
  \begin{axis}[
    boxplot/draw direction = x,
    y dir=reverse,
    height=1.75in,width=0.5\textwidth,
    xshift=2.0in,
    font=\footnotesize,
    xlabel = time (s),
    xticklabel style={%
       /pgf/number format/.cd,
           fixed,
           fixed zerofill,
           precision=0,
           },
    cycle list={[indices of colormap={4, 0, 5, 1, 6, 2, 3 of Set1-8}]},
    ytick = {1, 2, 3, 4, 5, 6, 7},
    yticklabel style = {align=center, font=\footnotesize}, %
    yticklabels = {,,,,,,},
    ]
    \foreach \m in {5, 1, 6, 2, 7, 3, 4} {
      \addplot+[boxplot, fill, draw=black] table[y=M\m] {\ubertime};
    }
  \end{axis}
  \end{tikzpicture}}
  \\
  \subfloat[Median fit (computed exactly) across 10 runs plotted for all methods.
  ]%
  {\label{fig:uber-runs}%
  \begin{tikzpicture}
  \begin{axis}[
    height=2in,width=0.95\textwidth,
    font=\footnotesize,    
    xlabel = time (s),
    ylabel = fit,
    ymin = 0.17, ymax = 0.191,
    yticklabel style={%
       /pgf/number format/.cd,
           fixed,
           fixed zerofill,
           precision=3,
           },
    unbounded coords=jump,
    legend pos = south east,
    legend style={font=\tiny},
    cycle list={    
      {index of colormap=4 of Set1-8,mark=*},
      {index of colormap=0 of Set1-8,mark=*}, 
      {yellow!90!gray},
      {index of colormap=1 of Set1-8,mark=square*},
      {index of colormap=6 of Set1-8,mark=*},
      {index of colormap=2 of Set1-8,mark=*},
      {index of colormap=3 of Set1-8,mark=*},}
    ]
    \foreach \m in {5, 1, 6, 2, 7, 3, 4} {
      \addplot+[thick,mark=none] table [x=TIME, y=M\m_MED]{\uberitp};
    }
    \legend{Random $s=2^{15}$,Hybrid $s=2^{15}$,Random $s=2^{16}$,Hybrid $s=2^{16}$,Random $s=2^{17}$,Hybrid $s=2^{17}$,Standard}
  \end{axis}
\end{tikzpicture}}
  
\caption{Comparison of \cprandlev (random and hybrid) with varying number of samples $s \in \set{2^{15},2^{16},2^{17}}$ and \cpals (standard)
  to compute a rank $r=25$ CP decomposition of the
  \textbf{Uber tensor} with 3.3 million nonzeros.
  Random uses $\tau=1$ and hybrid uses $\tau = 1/s$. Each experiment is run 10 times.}
  \label{fig:uber}
\end{figure}
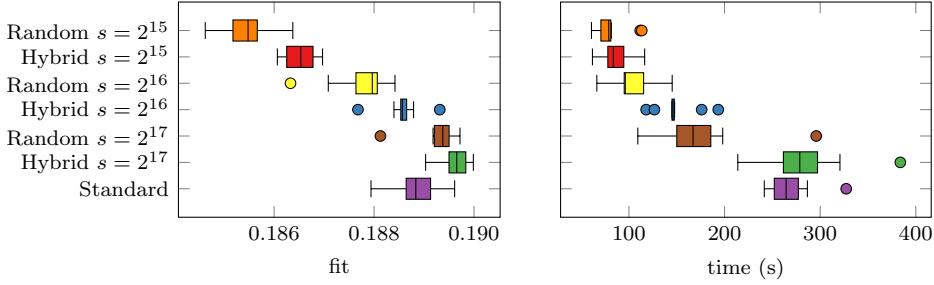
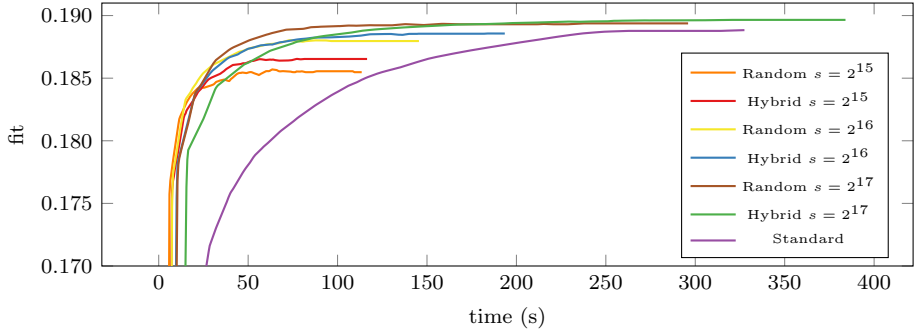 %

\Cref{fig:uber} shows results of ten run each for \cprandlev (random and hybrid) for sample sizes
$s \in \set{2^{15},2^{16},2^{17}}$ and \cpals (standard), run in computational environment (A).
\Cref{fig:uber-fit-time} presents box plots to compare the final fit and total
run time; \cref{fig:uber-runs} shows the median performance for each method;
individual runs are shown in \appuber.
For each value of $s$, hybrid deterministic sampling
improved the median final fit as compared to random, but hybrid is slower than random.
For $s=2^{17}$, both randomized algorithms
achieve a slightly better final fit than the standard method.
In this case, the random method is substantially faster than the standard method, but
the hybrid method is somewhat slower.
As this is a relatively small problem, we did not expect much improvement in time but instead use this
to demonstrate the correctness of the sampling method and how it improves as the number of samples increases.

\subsection{Order of magnitude speed improvement on massive sparse tensors}
\label{sec:amazon}

\pgfplotstabletranspose[input colnames to={IDs}, colnames from={IDs}]{\amazonraw}{data-amazon-fittrace-raw.csv}
\pgfplotstabletranspose[input colnames to={IDs}, colnames from={IDs}]{\amazonitp}{data-amazon-fittrace-interpl.csv}
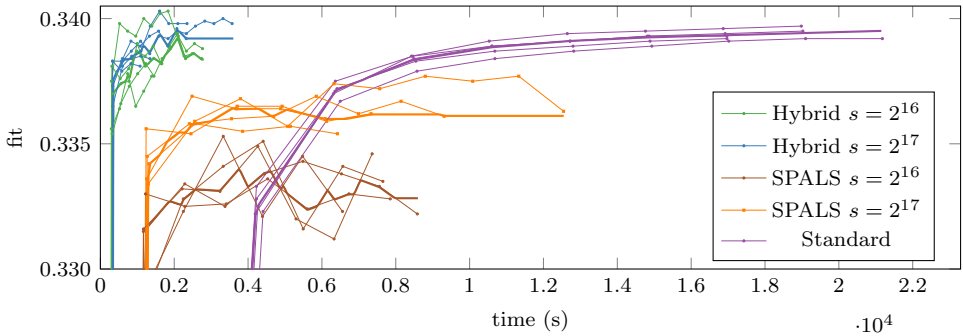
\begin{figure}
  \centering
    \subfloat[Statistics for 5 runs.  Total time and speedup do not include finding the true fit for runs of the random and hybrid methods, which was only done to compare to the standard method.]%
  {\label{tab:amazon-time}  \footnotesize
  \begin{tabular}{|c|c|c|c|c|c|}
    \hline
    & \multicolumn{1}{|c|}{\bf Mean}
    &
    & \multicolumn{1}{c|}{\bf Time Per} &
    \multicolumn{1}{c|}{\bf Median} &
    \multicolumn{1}{c|}{\bf Best} \\ 
    \multicolumn{1}{|c|}{\bf Method} &
    \multicolumn{1}{|c|}{\bf Time (s)} &
    \multicolumn{1}{c|}{\bf Speedup} &
    \multicolumn{1}{c|}{\bf Epoch (s)} &
    \multicolumn{1}{c|}{\bf Fit} &
    \multicolumn{1}{c|}{\bf Fit} \\ 
    \hline
Hybrid $s = 2^{16}$ & 2.21e+03 &  8.42 &  220.8 & 0.3384 & 0.3388 \\
Hybrid $s = 2^{17}$ & 2.13e+03 &  8.74 &  249.7 & 0.3392 & 0.3398 \\
Random $s = 2^{16}$ & 1.38e+03 & 13.46 &  209.5 & 0.3374 & 0.3380 \\
Random $s = 2^{17}$ & 1.47e+03 & 12.64 &  231.6 & 0.3385 & 0.3390 \\
SPALS $s = 2^{16}$  & 7.60e+03 &  2.45 & 1069.9 & 0.3349 & 0.3353 \\
SPALS $s = 2^{17}$  & 8.34e+03 &  2.23 & 1210.6 & 0.3369 & 0.3377 \\
\hline
Standard       & 1.86e+04 &  1.00 & N/A & 0.3393 & 0.3397 \\
    \hline
  \end{tabular}}
\\
\subfloat[Individual runs with (bias-corrected) estimated fit plotted for hybrid and true fit for standard and SPALS.
The dotted lines represent the individual runs with
  markers indicating epoch (5 iterations) for the randomized methods and one iteration for the standard method.
  The solid lines show  medians.]%
  {\label{fig:amazon-runs}\begin{tikzpicture}
  \begin{axis}[
    height=2in,width=\textwidth,
    font=\footnotesize,    
    xlabel = time (s),
    ylabel = fit,
    ymin = 0.33, ymax = 0.3405,
    xmin = 0.0,
    yticklabel style={%
       /pgf/number format/.cd,
           fixed,
           fixed zerofill,
           precision=3,
           },
    unbounded coords=jump,
    legend pos = south east,
    cycle list={
      {index of colormap=2 of Set1-8,mark=*},
      {index of colormap=1 of Set1-8,mark=*}, 
      {index of colormap=6 of Set1-8,mark=*}, 
      {index of colormap=4 of Set1-8,mark=square*},
      {index of colormap=3 of Set1-8,mark=*}}
    ]
    \foreach [count=\nn] \m in {1,2,5,6,7} {
      \pgfmathparse{\nn-1}
      \pgfplotsset{cycle list shift=\pgfmathresult}      
      \addplot+[thick,mark=none, forget plot ] table [x=TIME, y=M\m_MED]{\amazonitp};
    }
    \foreach \i in {1,...,5} {
      \foreach \m in {1,2,5,6,7} {        
        \addplot+[mark size=.4pt, very thin] table [x=M\m_R\i_Time, y=M\m_R\i_Fit]{\amazonraw};
      }
    }
    \legend{Hybrid $s=2^{16}$,Hybrid $s=2^{17}$,SPALS $s=2^{16}$,SPALS $s=2^{17}$,Standard}
  \end{axis}
\end{tikzpicture}}
 \caption{Comparison of \cprandlev (random and hybrid), SPALS, and \cpals (standard)
   to compute a rank $r=25$ CP decomposition of the \textbf{Amazon tensor} with 1.7 billion nonzeros.
   Each experiment is run 5 times. 
  }
  \label{fig:amazon}
\end{figure} %

\pgfplotstabletranspose[input colnames to={IDs}, colnames from={IDs}]{\redditraw}{data-reddit-fittrace-raw.csv}
\pgfplotstabletranspose[input colnames to={IDs}, colnames from={IDs}]{\reddititp}{data-reddit-fittrace-interpl.csv}
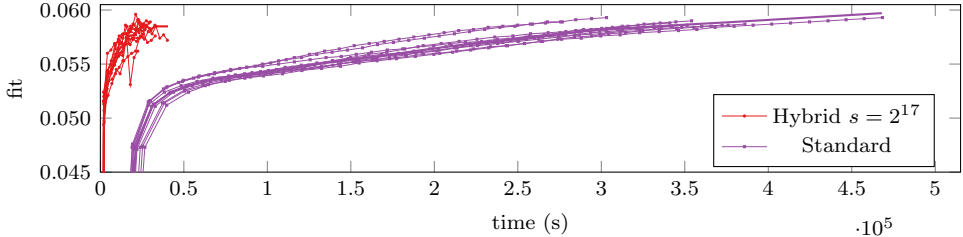
\begin{figure}
  \centering
    \subfloat[Median statistics and best fit across 10 runs.  Total time and speedup do not include finding the true fit for runs of the randomized methods, which was done to compare to the Standard method.]%
  {\label{tab:reddit-time}  \footnotesize
  \begin{tabular}{|c|c|c|c|c|c|}
    \hline
    & \multicolumn{1}{|c|}{\bf Mean}
    &
    & \multicolumn{1}{c|}{\bf Time Per} &
    \multicolumn{1}{c|}{\bf Median} &
    \multicolumn{1}{c|}{\bf Best} \\ 
    \multicolumn{1}{|c|}{\bf Method} &
    \multicolumn{1}{|c|}{\bf Time (s)} &
    \multicolumn{1}{c|}{\bf Speedup} &
    \multicolumn{1}{c|}{\bf Epoch (s)} &
    \multicolumn{1}{c|}{\bf Fit} &
    \multicolumn{1}{c|}{\bf Fit} \\ 
    \hline
    Random $s=2^{17}$ & $2.16 \times 10^4$ & 16.27 & 1832.6 & 0.0585 & 0.0590 \\ 
    Hybrid $s=2^{17}$ & $2.92 \times 10^4$ & 12.00 & 2231.0 & 0.0585 & 0.0589\\     
    \hline
    {Standard} & $3.51 \times 10^5$ & 1.00 & N/A & 0.0588 & 0.0593\\
    \hline
  \end{tabular}}
\\
\subfloat[Individual runs with the bias-corrected estimated fit plotted for \cprandlev and true fit plotted for \cpals.
We omit the random run because it is similar to hybrid.]%
  {\label{fig:reddit-runs}\begin{tikzpicture}
  \begin{axis}[
    height=1.5in,width=\textwidth,
    font=\footnotesize,    
    xlabel = time (s),
    ylabel = fit,
    ymin = 0.045, ymax = 0.0605,
    xmin = 0.0,
    scaled y ticks = false,
    yticklabel style={%
       /pgf/number format/.cd,
           fixed,
           fixed zerofill,
           precision=3,
           },
    unbounded coords=jump,
    legend pos = south east,
    cycle list={
      {index of colormap=0 of Set1-8,mark=*},
      {index of colormap=3 of Set1-8,mark=square*}}
    ]
    \foreach \m in {1,2} {
      \pgfmathparse{\m-1}      
      \pgfplotsset{cycle list shift=\pgfmathresult}
      \addplot+[thick,mark=none, forget plot ] table [x=TIME, y=M\m_MED]{\reddititp};
    }
    \foreach \i in {1,...,10} {
      \foreach \m in {1,2} {        
        \addplot+[mark size=.4pt, very thin] table [x=M\m_R\i_T, y=M\m_R\i_F]{\redditraw};
      }
    }
    \legend{Hybrid $s=2^{17}$,Standard}
  \end{axis}
 \end{tikzpicture}}
  \caption{Comparison of \cprandlev (random and hybrid) with number of samples $s=2^{17}$ and \cpals (standard)
   to compute a rank $r=25$ CP decomposition of the \textbf{Reddit tensor} with 4.68 billion nonzeros.  
   Random uses $\tau=1$ and hybrid uses $\tau = 1/s$. Each experiment is run 10 times. 
  The dotted lines represent the individual runs with
  markers indicating epoch (5 iterations) for the randomized methods and one iteration for the standard method.
  The solid lines show  medians.}
  \label{fig:reddit}
\end{figure} %

This section demonstrates how \cprandlev scales favorably for massive sparse tensors.
\Cref{fig:amazon} shows results for the Amazon tensor, run in computational environment (C),
and \cref{fig:reddit} shows results for the Reddit tensor, run in computational environment (B).
For both tensors, the \cprandlev runs (random and hybrid) use an estimated fit as described in \cref{sec:estimating-fit}
with  $s_{\text{fit}} = 2^{27}$ stratified samples, evenly divided between zeros and nonzeros,
The same sampled
entries are used \emph{across all runs} for consistent comparisons,
and the estimated fit plots are bias corrected
by the difference between the final true fit and final estimated fit, i.e., this difference was subtracted from all data points so that the final fit in the graph of an individual run is equal to the true fit.

The Amazon runs in \Cref{tab:amazon-time} show summary results from five runs each
of \cpals (standard) versus three randomized methods, \cprandlev (random and hybrid) and SPALS, each 
with $s \in \set{2^{16},2^{17}}$ samples.
For Amazon, %
a standard run requires approximately 5 hours.
Random (\cprandlev) is the fastest at less than 25 minutes, achieving a 13X speed-up as compared to
Standard (CP-ALS); however, the median and best fits are approximately 0.5\% worse than the standard method
for the lower number of samples.
Hybrid (\cprandlev) is second fastest at less than 37 minutes, with  an 8X speed-up as compared to the standard method and an improved final fit for $s=2^{17}$ samples.
It is a bit slower than the Random method because of the extra computations to determine $\pdet$ and do extra sampling, but the
advantage is higher fit values.
SPALS is the slowest randomized methods at 2 hours but still has a speedup of more than 2X compared to the standard method.
If we were to implement the ``practical'' improvements used for \cprandlev (combine repeats, hybrid sampling, estimated fit),
then we expect the timing for SPALS would be more competitive.
However, SPALS still has the lowest fit, more than 1\% worse than the standard method.
This is arguably not surprising given SPALS theoretically requires $\log n/\epsilon$ more samples
than \cprandlev.
Each method varies in the number of iterations required to terminate, so we also report the average
epoch time to facilitate comparison. %
We see that the epochs
are always fastest for the random method. 
\Cref{fig:amazon-runs} plots the times versus the fit for Hybrid (\cprandlev ), SPALS, and Standard (\cpals). We omit Random (\cprandlev) because the plots are very similar to Hybrid.

Recall that the complexity for a least squares solve for CP-ALS
with a sparse tensor is $O(\nnz(\X) r)$.
In comparison, the complexity of \cprandlev is $O(sr^2)$ for the QR factorization plus the cost to multiply $\Mx{Q}$-matrix of size $r \times s$  times the sparse matrix $\Mx[\tilde]{X}'$ of size $s \times n$, which costs $O(r\nnz(\Mx[\tilde]{X}))$.
So, we only achieve a speedup if $\max\set{sr,\nnz(\Mx[\tilde]{X})} \ll \nnz(\X)$.
In general, we do not know $\nnz(\Mx[\tilde]{X})$, but we have
some results on the number of nonzeros from \cref{fig:amazon-combine}, which yield
$\nnz(\X)/\nnz(\Mx[\tilde]{X}) \approx O(10)$, which is what we observe in terms of the speed-ups.

The Reddit runs in \Cref{tab:reddit-time} displays statistics for ten runs each
of \cpals (standard) and the randomized methods \cprandlev (random and hybrid)
with $s = 2^{17}$ samples.
For Reddit, we have 6 hours for random and 8 hours for hybrid methods versus 4 days for the standard method, yielding more than a 12X speedup.
As with Amazon, we also report the average
epoch time to facilitate comparison between the random and hybrid methods
and again see that the epochs
are always faster for the random method. 
The best fits are essentially equivalent across all methods.
\Cref{fig:reddit-runs} shows the estimated fit versus time for the hybrid method,
and the true fit versus time for standard method. We omit the random method because the plots look nearly identical to the hybrid method.

\subsection{Example factors for massive tensor}
\label{sec:reddit}

It is useful to show that tensor factors have a meaningful interpretation to prove these
computations are indeed worthwhile.
For this reason, this section shows several examples factors for the Reddit tensor as computed by \cprandlev. The complete set of factors is provided in \appreddit.
Reddit is a community forum wherein users comment within subreddits related to their interests.
The number of users associated with a subreddit can vary widely, with \texttt{r/AskReddit} and \texttt{r/funny} being two are the larger subreddits.
A few notes about the data processing are in order: common stop words were removed and the remaining worked were stemmed (e.g., ``people'' becomes ``peopl'');
users, subreddits, and words with fewer than five entries were removed.

We give some examples of the components computed for this tensor in \cref{fig:reddit-factor6,fig:reddit-factor18,fig:reddit-factor19}.
A component is the collection of matched matrix columns, i.e., component $j$ is $\Ak{1}(:,j)$, $\Ak{2}(:,j)$, and $\Ak{3}(:,j)$.
We cannot show the entirety of any component since the smallest dimension is 176k.
Instead, we show the top-25 highest-magnitude subreddits, the top-25 highest-magnitude words, and the top 1000 highest-magnitude users
as bar charts.
The length of a bar represents the magnitude and the color represents the \emph{overall} prevalence in the data, on a scale of zero to one.
In this manner, blue colors indicate rarer words or subreddits and are of interest since they are less likely to appear in many factors.
\begin{itemize}[leftmargin=0.5cm]
\item \Cref{fig:reddit-factor6} shows component 6 (of 25) which is focused on non-U.S. news. The word factor includes rarer words like countries (stemmed to ``countri'') and world.
One can see the top subreddits include ``worldnews'', ``europe'', ``unitedkingdom'', ``canada'', ``australia'', ``syriancivilwar'', ``india'', ``Israel'', ``UkraineConflict'', and ``Scotland''.
\item \Cref{fig:reddit-factor18} shows component 18 (of 25) focused on soccer and sports.
The top words include  ``player'', ``team'', ``leagu'' (stemmed version of league), ``goal'', ``fan'', and ``club''.
The top subreddits include ``soccer'', ``reddevils'', ``Gunners'', ``FIFA'', ``LiverpoolIFC'', etc.
\item \Cref{fig:reddit-factor19} shows component 19 (of 25) focused on movies and television, with a lean toward science fiction and fantasy.
  The top words include ``movi[e]'', ``film'', ``watch'', and ``charact[er]''.
  The top subreddits include ``movies'', ``television'', ``StarWars'', ``gameofthrones'', ``marvelstudios'', etc.
\end{itemize}

\begin{figure}
  \centering
  \includegraphics[width=0.73\textwidth]{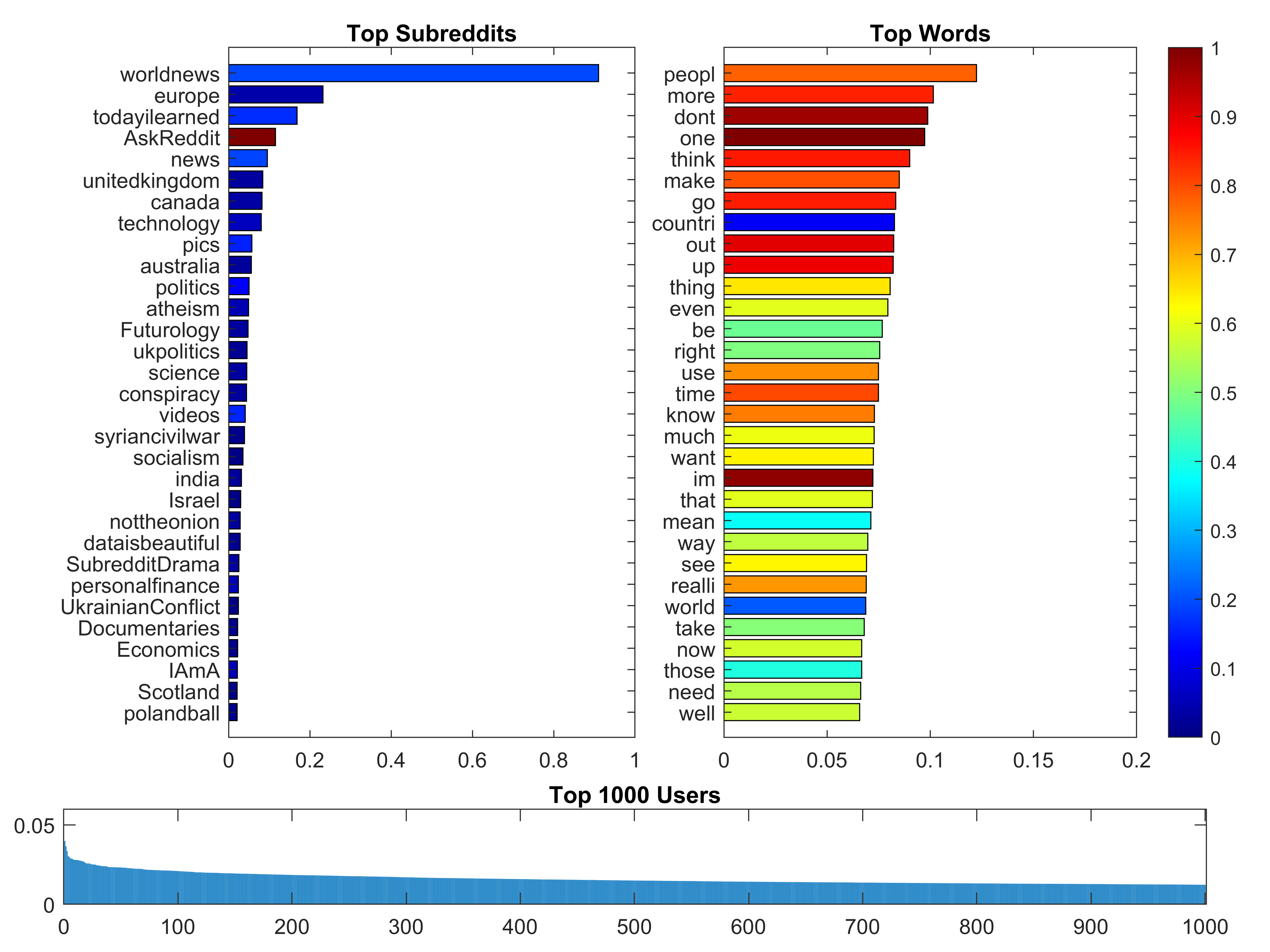}
  \caption{Reddit Factor 6/25: Politics and World News}
  \label{fig:reddit-factor6}
\end{figure} %

\begin{figure}
  \centering
  \includegraphics[width=0.73\textwidth]{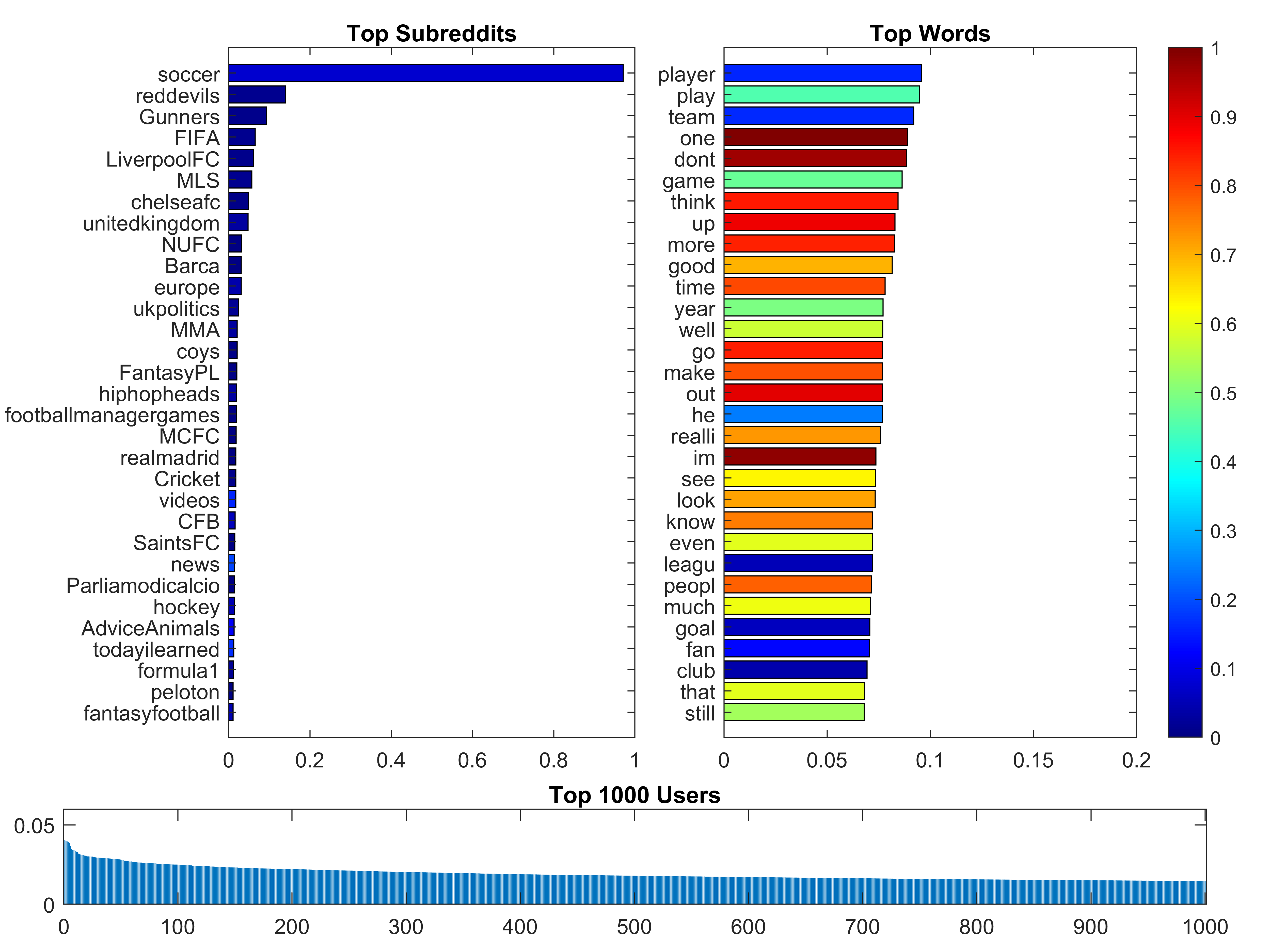}
  \caption{Reddit Factor 18/25: Soccer}
  \label{fig:reddit-factor18}
\end{figure} %

\begin{figure}  
  \centering
  \includegraphics[width=0.73\textwidth]{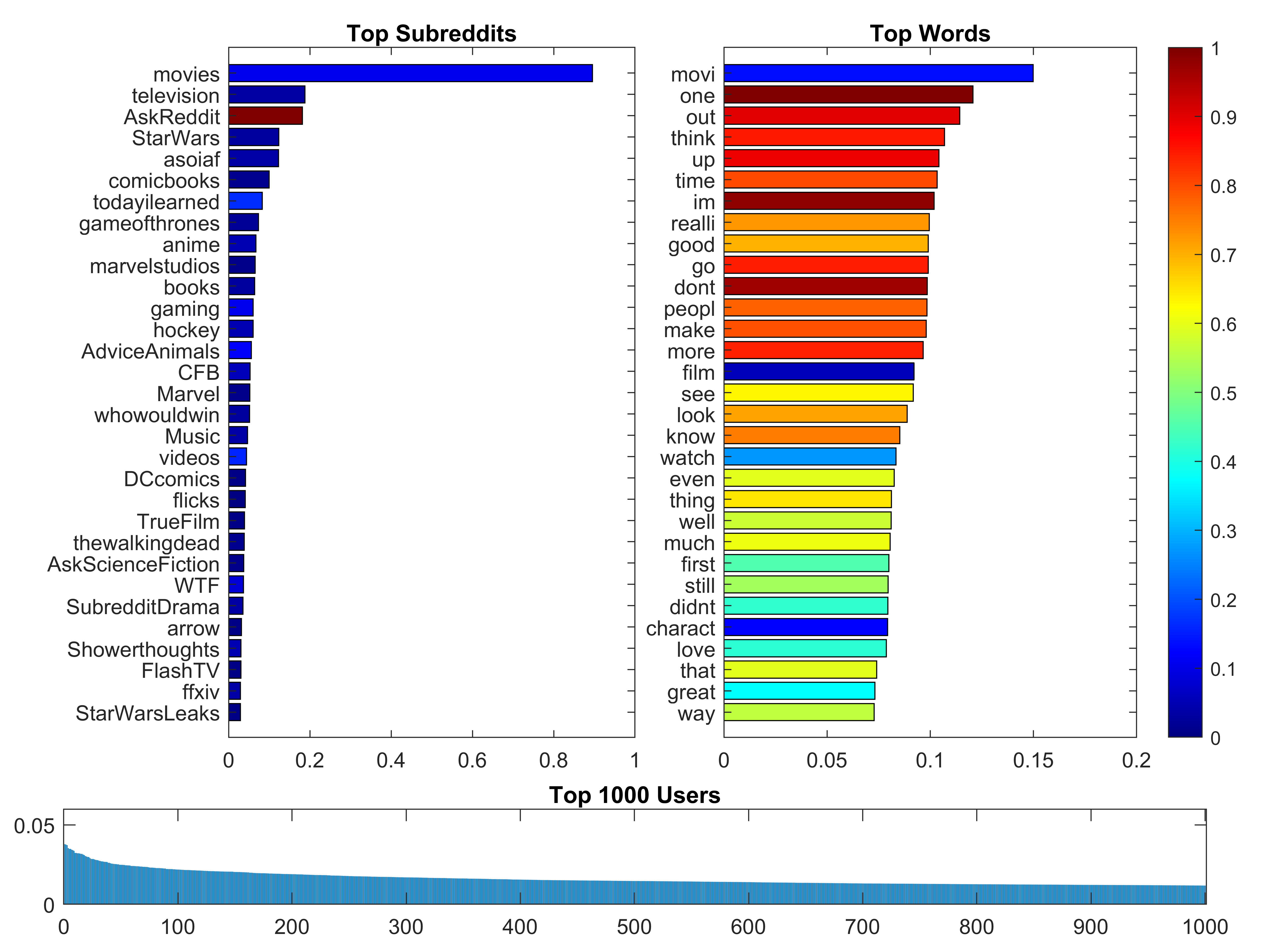}
  \caption{Reddit Factor 19/25: Film and Television}
  \label{fig:reddit-factor19}
\end{figure} %

\section{Conclusions}

We propose \cprandlev, a randomized algorithm which applies leverage score-based sketching to the overdetermined least squares problem in \cpals.
This approach offers an alternative to CP-ARLS-MIX \cite{BaBaKo18} which cannot be used on sparse tensors because the required FFTs destroy the tensor's sparsity.
A numerical comparison on dense tensors is provided in \appdense.
The strong performance of CP-ARLS-MIX may hint that its theoretical sample complexity can be improved.

Our proposed method has a better sample complexity than SPALS \cite{ChPePeLi16}, an alternate randomized algorithm.
Additionally, our experiments on the Amazon tensor with 1.7~billion nonzeros showed that \cprandlev was faster and achieved a better solution.

The improvements in speed come down in part to several practical improvements.
A few high-probability rows can result in excessive repeats in the sampled matrix.
We developed two methodologies that are generally useful for leverage-score matrix sketching.
First, we combine repeated rows so that the linear system to be solved is smaller.
Second, we have a method for deterministically including high-probability rows so that we achieve more unique samples overall.
Another issue is in checking convergence, which requires computing least squares error,
which has cost almost equivalent to solving the least squares problem.
Instead of computing the fit exactly, we borrow the technique of estimating the fit with a limited number of samples from \cite{BaBaKo18}.
In our numerical results, we show that \cprandlev implemented with all these techniques yields an order of magnitude speed ups on large-scale sparse tensors.

The paper leaves open many exciting theoretical directions.  
What is the optimal way to pick the number of samples (per mode even) and the deterministic threshold?
In general, these were chosen in this paper through numerical experiments.
Is it possible to show that hybrid sampling improves the $\beta$ factor in the leverage score estimates or to give a bound on the improvement in the $\epsilon$-accuracy?
And is there a more robust stopping condition for the algorithm than estimated fit?
Especially on the large tensors, obtaining a low-variance estimate of the fit required an extremely large number of samples.

Finally, \cprandlev has another advantage over \cpals in that it can be used on large distributed datasets.
Say one wanted to decompose a tensor that had to be stored across multiple nodes.
Each iteration of \cpals requires solving a system involving the entire tensor, but using \cprandlev one could store all the factor matrices on one node and sample based off the associated leverage scores.
The node could then gather the sampled fibers from the distributed tensor and solve the much smaller sampled system on one node.
Implementing this distributed algorithm and parallelizing much of the current implementation is a direction of future work.

\appendix

\section{Proof of \cref{thm:sketching}}%
\label{sec:proof}%
% ---- Inserted File ----
For ease of reference to existing literature, we use standard least squares notation as follows.
Consider the overdetermined matrix least squares problem defined by the \emph{design matrix} $\A \in \R^{N \times r}$,
with $N > r$ and  $\rank(\A) = r$, and the matrix $\Mx{B} \in \R^{N \times n}$.
Define the optimal squared residual to be
\begin{equation}
  \resid^2 \triangleq \min_{\Mx{X} \in \R^{r \times n}} \|\A \Mx{X} - \Mx{B}\|_F^2.
\end{equation}
The SVD of the design matrix is $\A = \UA \Mx{\Sigma}{\Mx{A}} {\Mx{V}{\!\!\Mx{A}}}^{\!\!\Tr}$,
so $\UA$ is an orthonormal basis for the $d$-dimensional column space of $\A$.
Let $\UA^{\perp}$ be an orthonormal basis for the $(N-r)$-dimensional
subspace orthogonal to the column space of $\Mx{A}$.
We define $\Bperp$ to be the projection of the the columns of $\Mx{B}$ onto this orthogonal subspace:
$\Bperp \triangleq \UA^{\perp} \UA^{\perp \trans} \Mx{B}$.
This matrix is important because the residual of the least squares problem is its Frobenious norm;
$\Mx{X}$ can be chosen so that each column in $\A \Mx{X}$ exactly matches the part of the corresponding column in $\Mx{B}$
in the column space of $\A$ but cannot, by definition, match anything in the range spanned by $\UA^{\perp}$:
\begin{equation*}
  \resid^2 = \min_{\Mx{X} \in \R^{r \times n}} \|\A \Mx{X} - \Mx{B}\|_F^2 = \|\UA^{\perp} \UA^{\perp \trans} \Mx{B}\|_F^2 = \|\Bperp\|_F^2
\end{equation*}
Denoting the solution to the least squares problem by by $\Xopt$
yields $\Mx{B} = \A \Xopt + \Bperp$.

Now consider the sketching problem defined by a matrix $\Mx{S} \in \Real^{s \times N}$:
\begin{equation}\label{eq:lsq}
  \min_{\Mx{X} \in \R^{r \times n}} \|\Mx{S}\A \Mx{X} - \Mx{S}\Mx{B}\|_F^2.
\end{equation}
Following the technique in Drineas et al.~\cite{drineas2011faster}, we split the proof into two parts.
In \cref{sec:prop-sketch-matr}, we prove bounds on both the residual and the solution of the sketched system
for a \emph{specific} sketching matrix $\Mx{S}$ that satisfies certain \emph{structural conditions}.
The proofs follow deterministically and do not consider the random aspect of the sketching matrix generation.
In \cref{sec:proof-that-leverage},
we then consider that $\Mx{S}$ is drawn from a distribution over matrices $\mathcal{D}$,
i.e., $\Mx{S} \sim \mathcal{D}$, and prove that the required structural conditions hold with high probability
if the number of samples is large enough.
Finally, the proof is completed by connecting these parts so
that the bounds on the residual and solution hold with high probability.

\subsection{Properties of sketching matrix under structural conditions}\label{sec:prop-sketch-matr}%
The main results mirror Lemma 1 and 2 in
\cite{drineas2011faster}.
The structure is also similar to Theorem 23 in
Woodruff~\cite{Wo14}, except that work uses
CountSketch, a different type of sketching.

We begin by assuming that our design matrix satisfies two \emph{structural conditions}:
\begin{align}
  \tag{SC1} \label{eq:sc1}
  \sigma^2_{\text{min}}(\Mx{S} \UA) &\geq 1/\sqrt{2},  \qtext{and}\\
  \tag{SC2} \label{eq:sc2}
  \|\UA^{\trans} \Mx{S}^{\trans} \Mx{S} \Bperp \|_F^2 &\leq \epsilon \resid^2/2. 
\end{align}

We first consider bounds with no constraints on the matrix $\Mx{B}$.
The first result is analogous to \cite[Lemma~1]{drineas2011faster}
except that we prove it for the \emph{matrix} least squares case.
We omit the proof because it follows the same logic as the vector case.
\begin{theorem}\label{thm:sketch-bound}
  For the overdetermined least squares problem \cref{eq:lsq}, assume
  the sketch matrix $\Mx{S}$ satisfies \cref{eq:sc1,eq:sc2} for some
  $\epsilon \in (0, 1)$.  Then the solution to the sketched problem,
  denoted $\Xtildeopt$, satisfies the following two bounds:
  \begin{align*}
    \|\A \Xtildeopt - \Mx{B} \|_F^2 &\leq (1 + \epsilon) \|\A \Xopt - \Mx{B} \|_F^2, \qtext{and} \\
    \|\Xopt - \Xtildeopt \|_F^2 &\leq \frac{\epsilon \|\A \Xopt - \Mx{B} \|_F^2}{\sigma^2_{\text{\rm min}}(\A)} .
  \end{align*}
\end{theorem}

We can obtain a tighter bound on the solution matrix if we assume a constant fraction of
 the columns of $\Mx{B}$ is in the column space of $\A$.  This is typically a
reasonable assumption for real-world least squares problems as the fit
is only practically interesting if this is true.
We again omit the proof because it follows the same logic as the vector case.

\begin{theorem}[\cite{drineas2011faster}]
  For the overdetermined least squares problem \cref{eq:lsq}, assume
  the sketch matrix $\Mx{S}$ satisfies \cref{eq:sc1,eq:sc2} for some $\epsilon \in (0, 1)$.
  Furthermore, assume that $\|\UA \UA^{\trans} \Mx{B} \|_F \geq \gamma \|\Mx{B}\|_F$ for some fixed $\gamma \in (0, 1]$.
  Then the solution to the sketched problem, denoted $\Xtildeopt$, satisfies the following bound:
  \begin{equation*}
    \|\Xopt - \Xtildeopt \|_F^2 \leq \epsilon^2 \kappa(\A)^2 (\gamma^{-2} - 1) \|\Xopt\|_F^2,
  \end{equation*}
  where $\kappa(\A)$ denotes the condition number of the matrix $\A$.
\end{theorem}

\subsection{Proof that sketching matrix meets structural conditions}\label{sec:proof-that-leverage}%
In this section, we show that the methodology for choosing the columns
via the leverage-score-based sampling scheme yields the desired bounds.
The first structural condition \cref{eq:sc1} can be shown as a corollary to the following result in Woodruff \cite{Wo14}:

\begin{lemma}[\cite{Wo14}]\label{lem:lev-scores}
  Consider $\Mx{A} \in \R^{N \times r}$, its SVD $\UA \Mx{\Sigma}_{\Mx{A}} \Mx{V}^{\trans}_{\Mx{A}}$,
  and row leverage scores $\ell_i(\A)$.
  Let $\overline{\Vl}(\A)$ be an overestimate of the leverage score such that for some positive $\beta \leq 1$,
  we have $p_i\big(\overline{\Vl}(\A) \big) \geq \beta \cdot p_i\big(\Vl(\A)\big)$ for all $i \in [N]$.
  Construct row sampling and rescaling matrix $\Mx{S} \in \R^{s \times N}$
  by importance sampling according to the leverage score overestimates, $\overline{\Vl}(\A)$.
  If $s > 144 r \ln(2r/\delta)/(\beta \epsilon^2)$, then the following holds
  with probability at least $1 - \delta$ simultaneously for all $i$:
  $1 - \epsilon \leq \sigma_i^2(\Mx{S} \UA) \leq 1 + \epsilon$.
\end{lemma}

\noindent Fixing $\epsilon = 1 - 1/\sqrt{2}$ in \cref{lem:lev-scores} yields the  Corollary  we require.

\begin{lemma}\label{lem:sc1}
  Consider $\Mx{A} \in \R^{N \times r}$, its SVD $\UA \Mx{\Sigma}_{\Mx{A}} \Mx{V}^{\trans}_{\Mx{A}}$,
  and row leverage scores $\ell_i(\A)$.
  Let $\overline{\Vl}(\A)$ be an overestimate of the leverage score such that for some positive $\beta \leq 1$,
  we have $p_i\big(\overline{\Vl}(\A) \big) \geq \beta \cdot p_i\big(\Vl(\A)\big)$ for all $i \in [N]$.
  Construct row sampling and rescaling matrix $\Mx{S} \in \R^{s \times N}$
  by importance sampling according to the leverage score overestimates, $\overline{\Vl}(\A)$.
    If $s > C r \ln(2r/\delta)/\beta $ with $C = 144/(1-1/\sqrt{2})^2$,
  then $\sigma_{\min}^2(\Mx{S} \UA) \geq 1/2$
  with probability at least $1 - \delta$.
\end{lemma}

The second structural condition \cref{eq:sc2} can be proven using results for randomized matrix-matrix multiplication.
Consider the matrix product $\UA^\trans \Bperp$.
This projects the part of  the columns of $\Mx{B}$ outside of the column space of $\A$ onto the column space
of $\A$ and thus by definition is equal to the all zeros  matrix $\Mx{0}_{r \times n}$ (we have assumed $\rank(\A) = r$).
This condition requires us to bound how well the sampled product
$\UA^\trans \Mx{S}^{\trans} \Mx{S} \Bperp$ approximates the original product.
We can do this via the following lemma from Drineas, Kannan, and Mahoney \cite{drineas2006fast}.

\begin{lemma}[\cite{drineas2006fast}]\label{lemma:randMatMult}
  Consider two matrices of the form $\Mx{A} \in \R^{n \times m}$ and $\Mx{B} \in \R^{n \times p}$
  and let $s$ denote the number of samples.
  We form an approximation of the product $\Mx{A}' \Mx{B}$ as follows.
  Choose $s$ rows, denoted $\{\xi^{(1)}, \ldots, \xi^{(s)}\}$, according to the probability
  distribution defined by $\V{p} \in [0,1]^n$ with the property that
  there exists $\beta >0$ such that
  $p_k \geq \beta \|\A(k,:)\|^2 / \|\A\|_F^2 \qtext{for all} k \in [n]$.
  Then form the approximate product
  \begin{equation*}
    \frac{1}{s} \sum_{t = 1}^{s} \frac{1}{p_{\xi^{(t)}}} \Mx{A}(\xi^{(t)},:)^{\trans} \Mx{B}(\xi^{(t)}, :) \triangleq (\Mx{S} \A)^{\trans} \Mx{S} \Mx{B},
  \end{equation*}
  where we define $\Mx{S}$ to be the random row sampling and rescaling operator.
  We then have the following guarantee on the quality of the approximate product:
  \begin{equation*}
    \E \left [  \| \Mx{A}^{\trans} \Mx{B} - (\Mx{S} \A)^{\trans} \Mx{S} \Mx{B} \|_F^2 \right ] \leq \frac{1}{\beta s} \|\Mx{A}\|_F^2 \|\Mx{B}\|_F^2.
  \end{equation*}
\end{lemma}

 \noindent We can apply \cref{lemma:randMatMult} to bound the probability of \cref{eq:sc2} holding.

\begin{lemma}\label{lem:sc2}
  Consider full rank $\Mx{A} \in \R^{N \times r}$, its SVD $\UA \Mx{\Sigma}_{\Mx{A}} \Mx{V}^{\trans}_{\Mx{A}}$, and row leverage scores $\ell_i(\A)$.
  Define the probability distribution $\V{p} \in [0,1]^n$ and assume there exists $\beta \in (0,1]$
  such that $p_i \geq \beta \ell_i(\A)/d$ for all $i \in [N]$.
  Construct row sampling and rescaling matrix $\Mx{S} \in \R^{s \times N}$ by importance sampling by the leverage score overestimates.
  Then provided $s \geq  \frac{ 2 r}{\beta \delta \epsilon }$,
  the property $\|\UA^{\trans} \Mx{S}^{\trans} \Mx{S} \Bperp \|_F^2 \leq \epsilon \resid^2/2$
  holds with probability $\delta$.
\end{lemma}
\begin{proof}
  Apply \cref{lemma:randMatMult} to obtain a bound on the expected value:
  \begin{align*}
    \E \left [  \| \UA^\trans \Mx{S}^{\trans} \Mx{S} \Bperp \|_F^2 \right ]
    &= \E \left [  \| \Mx{0}_{r \times n} - \UA^\trans \Mx{S}^{\trans} \Mx{S} \Bperp \|_F^2 \right ], \\
    &= \E \left [  \| \Mx{\UA} \Mx{\Bperp} - \UA^\trans \Mx{S}^{\trans} \Mx{S} \Bperp \|_F^2 \right ], \\
    & \leq \frac{1}{\beta s} \|\UA\|_F^2 \|\Bperp\|_F^2 = \frac{r}{\beta s} \|\Bperp\|_F^2
    \mbox{}=\frac{r}{\beta s} \mathcal{R}^2.
  \end{align*}
  Markov's inequality states that for non-negative random variable $X$ and scalar $t > 0$, we can bound the probability that $X \geq t$ as $\Prob [ X \geq t] \leq \E[X]/t$.
  We can apply this inequality to bound the probability that the sketching matrix violates \cref{eq:sc2}:
  \begin{align*}
    \Prob_{\Mx{S} \sim \mathcal{D}} \left [  \| \UA^\trans \Mx{S}^{\trans} \Mx{S} \Bperp \|_F^2  \geq \frac{\epsilon \|\Bperp\|_F^2}{2} \right ] &\leq \frac{ 2 \E \left [  \| \UA^\trans \Mx{S}^{\trans} \Mx{S} \Mx{\Bperp} \|_F^2 \right ]}{\epsilon \|\Bperp\|_F^2} \leq \frac{ 2 r}{\beta \epsilon s}
  \end{align*}
  where in the last step we have used our bound the expected value. Thus if we set the right-hand side equal to $\delta$, we obtain that the probability that \cref{eq:sc2} holds is greater than or equal to $1 - \delta$ as desired.  Solving for $s$ yields that we thus must have $s \geq  \frac{ 2 r}{\beta \delta \epsilon }$.
\end{proof}

\subsection{Main Theorem} We combine the above results to prove \Cref{thm:sketching} in the main text, here written in the standard least squares notation.

\begin{theorem}%
  \label{thm:sketching-app}
  Consider the least squares problem
  $\min_{\Mx{X} \in \Real^{r \times n}} \| \A \Mx{X} - \B\|^2$
  where $\A \in \Real^{N \times r}$ with $r \ll N$ and $\rank(\A) = r$
  and  $\B \in \Real^{N \times n}$.
  Let $\Vp \in [0,1]^N$ be a probability distribution and
  assume there exists a fixed
  $\beta \in (0,1]$ such that
  \begin{displaymath}
    \beta \leq \min_{i \in [N]} \frac{p_i r}{\ell_i(\A)} \qtext{for all} i \in [N].
  \end{displaymath}  
  For any  $\epsilon, \delta \in (0,1)$,
  set 
  \begin{displaymath}
    s=({r}/{\beta}) \max \set{ C \log (r/\delta), {1}/(\delta \epsilon) }
    \qtext{where}
    C = 144/(1-1/\sqrt{2})^2    ,
  \end{displaymath}
  and let $\Mx{S} = \RandSample(s,\Vp)$.
  Define 
  $\Xopt \equiv \arg \min_{\Mx{X} \in \Real^{r \times n}} \| \A \Mx{X} - \B\|^2$.
  Then $\Xtildeopt \equiv \arg \min_{\Mx{X} \in \Real^{r \times n}} \| \Mx{S} \A \Mx{X} - \Mx{S} \B\|_F^2$
  satisfies $\| \A \Xtildeopt - \B \|_F^2 \leq (1+\epsilon) \| \A \Xopt - \B \|_F^2$
  with probability at least $1-\delta$.
\end{theorem}

\begin{proof}
  Applying \cref{lem:sc1}, we have that \cref{eq:sc1} holds with probability $1-\delta/2$ if
  $s=C r \log (r/\delta)/\beta$.
  Applying \cref{lem:sc2}, we have that \cref{eq:sc2} holds with probability $1-\delta/2$ if
  $s=r/(\beta \delta \epsilon)$.
  Hence, a union bound says that \cref{eq:sc1,eq:sc2} both hold with probability $1-\delta$ if
  $s= (r/\beta) \max\set{ C \log(r/\delta), 1/(\delta\epsilon)}$.
  Combining this with \cref{thm:sketch-bound} yields the result.
\end{proof}

%
%

%

%
%
%
%
% --- End Inserted File ---

\section{Comparison of Sampling Methods}%
\label{sec:sampling}%
% ---- Inserted File ----
%
\Cref{fig:uber-sample-comparison} compares \cprandlev (random) using three different sampling methods: uniform sampling, length squared sampling, and leverage score sampling.
As with leverage score sampling, we use an the analogous bound on the $\ell_2$-norm squared of the rows of the KRP.
Let $\Z = \Ak{d} \odot \cdots \odot \Ak{1} \in \Real^{N \times r}$.
Letting $(i_1,\dots,i_d)$ be the multi-index corresponding to $i$
as defined in \cref{eq:bijection}, the $\ell_2$-norm squared can be bounded as $\| \Z(i, :) \|_2^2 \leq \prod_{k=1}^d \|\Ak(i_k, :)\|_2^2$.
The experiments are run with the same setup as \Cref{sec:uber} on environment (A).  We find a $r=25$ CP tensor decomposition, compute the true fit every epoch of 5 iterations, and terminate after the fit fails to improve by more than $10^{-4}$ for 3 epochs.
From \Cref{fig:uber-sample-comparison},
we can see that uniform sampling is not competitive.
Length-squared sampling is arguably competitive,
but the fit is not as high as leverage-score sampling.

\pgfplotstabletranspose[input colnames to={IDs}, colnames from={IDs}]{\uberrawsamp}{data-uber-sample-comparison-fittrace-raw.csv}
\pgfplotstabletranspose[input colnames to={IDs}, colnames from={IDs}]{\uberitpsamp}{data-uber-sample-comparison-fittrace-interpl.csv}
\begin{figure}
  \centering
  \pgfplotsset{
    height=1.5in,
    width=0.49\textwidth,
    xlabel = time (s),
    ylabel = fit,
    ymin = 0.05, ymax = 0.200,
    xmin = 0,
    xmax = 250,
    yticklabel style={%
      /pgf/number format/.cd,
      fixed,
      fixed zerofill,
      precision=2,
    },
    unbounded coords=jump,
    legend pos = south east,
  }
  \subfloat[$s=2^{16}$]{\label{fig:supp-uber-sample-comp16}%
  \begin{tikzpicture}
      \begin{axis}[font=\footnotesize,    
          cycle list={
            {index of colormap=0 of Set1-8,mark=*},
            {index of colormap=1 of Set1-8,mark=*},
            {index of colormap=2 of Set1-8,mark=*}}
          ]
        \foreach \m in {1, 3, 5} {
          \pgfmathparse{(\m-1)/2}      
          \pgfplotsset{cycle list shift=\pgfmathresult}
          \addplot+[thick,mark=none, forget plot ] table [x=TIME, y=M\m_MED]{\uberitpsamp};
        }
        \foreach \i in {1,...,10} {
          \foreach \m in {1, 3, 5} {        
            \addplot+[mark size=.4pt, very thin] table [x=M\m_R\i_T, y=M\m_R\i_F]{\uberrawsamp};
          }
        }
        \legend{Uniform,Length Squared,Leverage Scores}
      \end{axis}
    \end{tikzpicture}
  }
  \subfloat[\phantom{foo} $s=2^{17}$ \phantom{foo}]{\label{fig:supp-uber-sample-comp17}
    \begin{tikzpicture}
      \begin{axis}[font=\footnotesize,    
        cycle list={
          {index of colormap=0 of Set1-8,mark=*},
          {index of colormap=1 of Set1-8,mark=*},
          {index of colormap=2 of Set1-8,mark=*}}
        ]
        \foreach \m in {2, 4, 6} {
          \pgfmathparse{(\m-1)/2}       
          \pgfplotsset{cycle list shift=\pgfmathresult}
          \addplot+[thick,mark=none, forget plot ] table [x=TIME, y=M\m_MED]{\uberitpsamp};
        }
        \foreach \i in {1,...,10} {
          \foreach \m in {2, 4, 6} {        
            \addplot+[mark size=.4pt, very thin] table [x=M\m_R\i_T, y=M\m_R\i_F]{\uberrawsamp};
          }
        }
        \legend{Uniform,Length Squared ,Leverage Scores}
      \end{axis}
    \end{tikzpicture}
  }
  
\caption{Comparison of \cprandlev (random, $\tau = 1$) using different sampling methods (uniform, length squared, and leverage scores) 
  to compute a rank $r=25$ CP decomposition of the
  Uber tensor with 3.3 million nonzeros.
  The dotted lines represent the individual runs with
  markers indicating epoch (5 iterations).}
  \label{fig:uber-sample-comparison}
\end{figure}
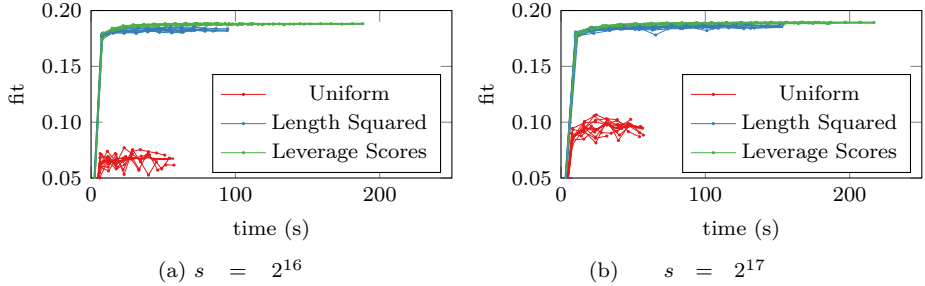 %

%
%
%
%
% --- End Inserted File ---

\section{Comparison on Dense Tensor}%
\label{sec:dense}%
% ---- Inserted File ----
%
%
%
We consider the performance of \cprandlev (random, i.e., $\tau=1$)
on a synthetic dense tensor as compared to
both the standard CP-ALS method and \cprand \cite{BaBaKo18}.
We consider \cprand with mixing (-MIX) and plain uniform sampling without mixing (-UNI).
Only the method with mixing is guaranteed to work.
We show only the results of \cprandlev using the random method ($\tau=1$) since the hybrid method ($\tau=1/s$) is similar.

\Cref{fig:syn_dense} plots total run time on the x-axis and the \emph{factor match score}\footnote{%
  Because this problem is synthetic, we know the true factor matrices and therefore can report the factor match score
  which says how similar the obtained solution is to the true solution; a factor match score of 1.0 is a perfect match.
}
corresponding to the best fit of the ten runs on the y-axis.
We describe the specifics of the data generation in \cref{sec:data-generation} and the specifics of the experimental setup in \cref{sec:experimental-setup}.
The tensor has been constructed so that the subproblems are coherent (difficult for \cprand-UNI),
which has the side effect of making several of the factors highly collinear (difficult for CP-ALS).
As a result, CP-ALS only achieves a factor match score around 0.814, and all the other methods achieve better
scores in less time (\cprandlev with $s=2^{9}$, \cprand-MIX with $s=2^8$, \cprand-UNI with $s=2^{10}$).
\cprand-UNI actually gets worse for $s=2^{11}$ samples.
Between \cprandlev and \cprand-MIX, the latter had superior performance, achieving higher scores in less time, and this is further emphasized by looking at the fraction of runs that achieved a score above 0.93.
\cprand-MIX was successful, i.e. achieved a score greater than or equal to 0.93,
80\% of time with $s=2^9$ samples, 70\% with $s=2^{10}$ samples, and 30\% with $s=2^{11}$ samples.
\cprandlev was only successful 10\% of the time with $s=2^{11}$ samples.
Thus, while \cprandlev is a viable choice for dense tensors, the experiments give the advantage to \cprand-MIX.

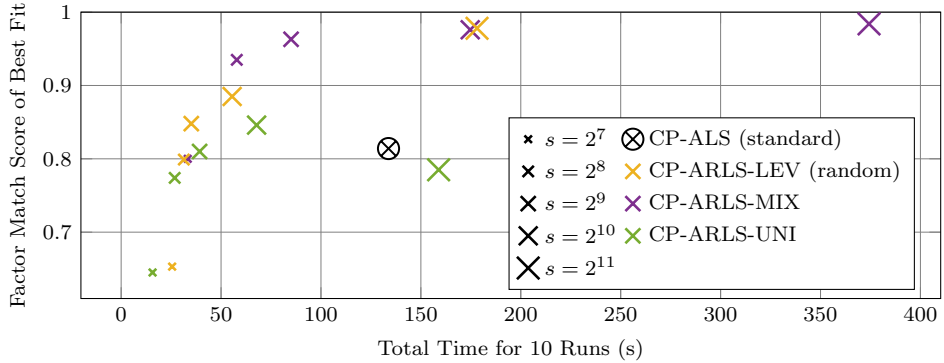
\begin{figure}
  \definecolor{mycolor1}{rgb}{0.0000,0.4470,0.7410}
  \definecolor{mycolor2}{rgb}{0.8500,0.3250,0.0980}
  \definecolor{mycolor3}{rgb}{0.9290,0.6940,0.1250}
  \definecolor{mycolor4}{rgb}{0.4940,0.1840,0.5560}
  \definecolor{mycolor5}{rgb}{0.4660,0.6740,0.1880} 
    \pgfplotsset{
      legend style={font=\footnotesize},
    scale only axis,
    height = 1.5in,
    width = 4.5in,
    xlabel = Total Time for 10 Runs (s),
    ylabel = Factor Match Score of Best Fit,
    legend pos=south east,
    legend columns=2,
    legend cell align={left},
    scatter,scatter src=explicit symbolic, 
    scatter/classes={
      a1={mark=otimes,black, mark size=4, line width = 0.75 pt},b1={mark=otimes,black, mark size=4, line width = 0.75 pt},c1={mark=otimes,black,mark size=4, line width = 0.75 pt},d1={mark=otimes,black,mark size=4, line width = 0.75 pt},e1={mark=otimes,black,mark size=4, line width = 0.75 pt},
      a3={mark=x,mark size=2, mycolor3},b3={mark=x,mark size=3,mycolor3},c3={mark=x,mark size=4,mycolor3},d3={mark=x,mark size=5,mycolor3},e3={mark=x,mark size=6,mycolor3},
      a4={mark=x,mark size=2, mycolor4},b4={mark=x,mark size=3,mycolor4},c4={mark=x,mark size=4,mycolor4},d4={mark=x,mark size=5,mycolor4},e4={mark=x,mark size=6,mycolor4},
      a5={mark=x,mark size=2, mycolor5},b5={mark=x,mark size=3,mycolor5},c5={mark=x,mark size=4,mycolor5},d5={mark=x,mark size=5,mycolor5},e5={mark=x,mark size=6,mycolor5},
      m={black}},  
    }
    \centering
    \begin{tikzpicture}
      \begin{axis}[font=\footnotesize,ymax=1,grid style={line width=.1pt, draw=gray},  grid=both,]
        \addlegendimage{only marks, mark=x, mark size=2pt, line width = 1 pt}
        \addlegendentry{$s=2^7$}
        \addlegendimage{only marks, mark=otimes, mark size=4pt, line width = 0.75 pt}
        \addlegendentry{CP-ALS (standard)}
        \addlegendimage{only marks, mark=x, mark size=3pt, line width = 1 pt}
        \addlegendentry{$s=2^8$}
        \addlegendimage{only marks, mark=x, mark size=4pt, line width=1pt, mycolor3}
        \addlegendentry{CP-ARLS-LEV (random)}
        \addlegendimage{only marks, mark=x, mark size=4pt, line width = 1 pt}
        \addlegendentry{$s=2^9$}
        \addlegendimage{only marks, mark=x, mark size=4pt, line width=1pt, mycolor4}
        \addlegendentry{CP-ARLS-MIX}
        \addlegendimage{only marks, mark=x, mark size=5pt, line width = 1 pt}
        \addlegendentry{$s=2^{10}$}
        \addlegendimage{only marks, mark=x, mark size=4pt, line width=1pt, mycolor5}
        \addlegendentry{CP-ARLS-UNI}
        \addlegendimage{only marks, mark=x, mark size=6pt, line width = 1 pt}
        \addlegendentry{$s=2^{11}$}
        \addplot[only marks, line width = 1 pt] table[meta=meta] {data-synthetic-dense-order3.dat};
       \end{axis}
    \end{tikzpicture}
    \caption{Comparing methods on a $500 \times 500 \times 500$ synthetic dense tensor engineered such that the factor matrices have a few rows with concentrated leverage scores.  Each method is run 10 times from random initializations (the same 10 initializations were used across all methods).}
    \label{fig:syn_dense}
\end{figure} %

\subsection{Data generation}
\label{sec:data-generation}

For our experiment, we used a synthetic 500 $\times$ 500 $\times$ 500 tensor.  We begin by generating three factor matrices of size $500 \times 25$ with independent standard Gaussian entries.  
The first three columns of each factor matrix are set to 0 and then seeded with a small number of non-zeros based on user-specified parameters: spread and magnitude.
The spread specifies how many non-zeros elements are added to each of these first three columns; these non-zero elements are chosen such that each row has no more than one non-zero in the first three columns.
We use a spread of 5 such the first five rows are non-zero in the first column, the second set of five rows are non-zero in the second column, and the third set of five rows are non-zero in the third column.
The leverage scores corresponding to each column's contribution to the column space are thus spread over at most 5 rows, giving 15 total rows with very high leverage score.
The magnitude specifies the size of these non-zero elements (set to 3 in our experiments) and then has a small amount of independent noise added chosen uniformly between 0 and 0.05.
Finally, we construct the associated rank 25 tensor and add 5\% Gaussian noise to the elements.

\subsection{Experimental setup}
\label{sec:experimental-setup}

The experiments were run on environment (A) using MATLAB (Version 2018a) and the Tensor Toolbox for MATLAB \cite{TTB_Software}.
CP-ALS used the default settings except that $\mathtt{tol}=10^{-5}$ and the maximum number of iterations was set to 250.
For all the randomized algorithms, the epoch size $\eta$ was set to 5 and the algorithm terminated once $\pi = 5$ epochs failed to improve the estimated fit by $\texttt{tol} = 10^{-4}$.
We created 10 random initializations and used the same initializations for \emph{all} experiments.
Each algorithm and sample combination were run from these 10 random initialization and the total time of all 10 runs was recorded.
We computed the factor match score for the run that had the highest fit value.

%
%
%
%
% --- End Inserted File ---

\section{End-to-end Complexity of Algorithm}
\label{sec:complexity}
% ---- Inserted File ----
%
%
Here we summarize the costs of the CP-ARLS-LEV.
Recall that the tensor $\X$ is of order $(d+1)$ and size $n_1 \times n_2 \times \cdots \times n_{d+1}$,
the target rank is $r$, the number of samples per least squares solve is $s$,
and the number of samples used to (optionally) estimate the fit is $\sfit$.
We assume $d$ is a constant for the purposes of complexity calculations,
and let $\bar n = \max_k n_k$.
For a sparse tensor, $\nnz(\X)$ is the number of nonzeros in $\X$, with $\nnz(\X) = O(\prod_k n_k)$ if $\X$ is dense.
Line numbers refer to \cref{alg:cp-arls-lev} unless otherwise noted.

\subsection{Preprocessing}
\label{sec:preprocessing}

We get a speedup in the extraction of the fibers from the tensor $\X$
by precomputing linear indices for each unfolding and for each nonzero (see~\cref{sec:effic-sampl-from})
at a one-time cost of $O(\nnz(\X))$.

We also need to compute the leverage scores. This is computed
for the initial factor matrices at a cost of $O(r^2 \bar n)$
for all the factor matrices.

\subsection{Per-epoch costs}
\label{sec:per-epoch-costs}

Each epoch comprises $\eta$ outer iterations, with costs detailed below.
The only other cost per epoch is computing the fit.
Computing the exact fit costs $O(r \nnz{\X})$, which can easily the dominant cost
for larger problems, even if it is only computed once per epoch.
For the sparse problems with billions of nonzeros, an estimated fit is used as described in \cref{sec:estimating-fit} and costs $O(r \sfit)$ operations.

\subsection{Per-iteration costs}
\label{sec:per-iteration-costs}

Each outer iteration comprises $d+1$ inner iterations, indexed by $k$.
Inner iteration $k$ has the following costs:
\begin{itemize}
\item Computing the sampled indices and weights (Line 8):
  If $\tau=1$ (corresponding to what we refer to as random sampling in the experiments),
  generating indices in Line 5 of \cref{alg:skrpidx}
  is the main computation and costs $O(s\bar n)$ operations.
  If $\tau<1$, there is also the cost of computing the deterministic indices,
  with a worst-case of $O(s^d)$ computations; however, this is an optional step that can
  be aborted if it proves overly expensive.
  Combining repeated indices in Line 5 of \cref{alg:skrplev}
  requires $O(s \log s)$ computations to sort and find duplicates.
\item Computing the sampled KRP (Line 9): $O(sr)$
\item (\textbf{Dominating memory movements}) Computing the sampled tensor (Line 10): This requires finding nonzeros corresponding to the sampled rows using the precomputed linear indices. The main price we pay here is memory movement to extract up to $O(s\bar n)$ nonzeros.
\item (\textbf{Dominating computations}) Solving the sampled least squares system (Line 11): This is the solution of an $s \times r$ system with $\bar n$ right-hand sides,
  so the work is $O(s r^2)$ operations to form the QR factorization $O(sr\bar n)$ operations to multiply the right-hand sides by the $\Mx{Q}$-matrix.  If the right-hand side is sparse, the $\Mx{Q}$-matrix can be applied via sparse matrix vector multiplication (SpMV) and the cost becomes $O(r \nnz(\Mx[\tilde]{X}))$ where $\Mx[\tilde]{X}$ is the flattened sampled tensor.
\item Computing updated sampling probabilities (line 12): $O(r^2 \bar n)$ to compute the QR factorization of $\Ak$
\end{itemize}

%
%
%
%
% --- End Inserted File ---

\section{Detailed Runs on Uber Tensor}
\label{sec:uber-details}
% ---- Inserted File ----
%
%
%
%
\Cref{fig:supp-uber} provides the full run data for the experiments discussed in \cref{sec:uber} and summarized in \cref{fig:uber}.
The dotted lines are the individual runs and the solid lines are the median over the interpolated runs.
The standard method is included %
for ease of comparison.
Each marker for random and hybrid represents one epoch of five outer iterations, and each marker for \cpals (standard)
corresponds to one outer iteration.

\pgfplotstabletranspose[input colnames to={IDs}, colnames from={IDs}]{\uberraw}{data-uber-fittrace-raw.csv}
\pgfplotstabletranspose[input colnames to={IDs}, colnames from={IDs}]{\uberitp}{data-uber-fittrace-interpl.csv}
\pgfplotstabletranspose[input colnames to={IDs}, colnames from={IDs}]{\ubertime}{data-uber-totaltimes.csv}
\pgfplotstabletranspose[input colnames to={IDs}, colnames from={IDs}]{\uberfit}{data-uber-finalfits.csv}

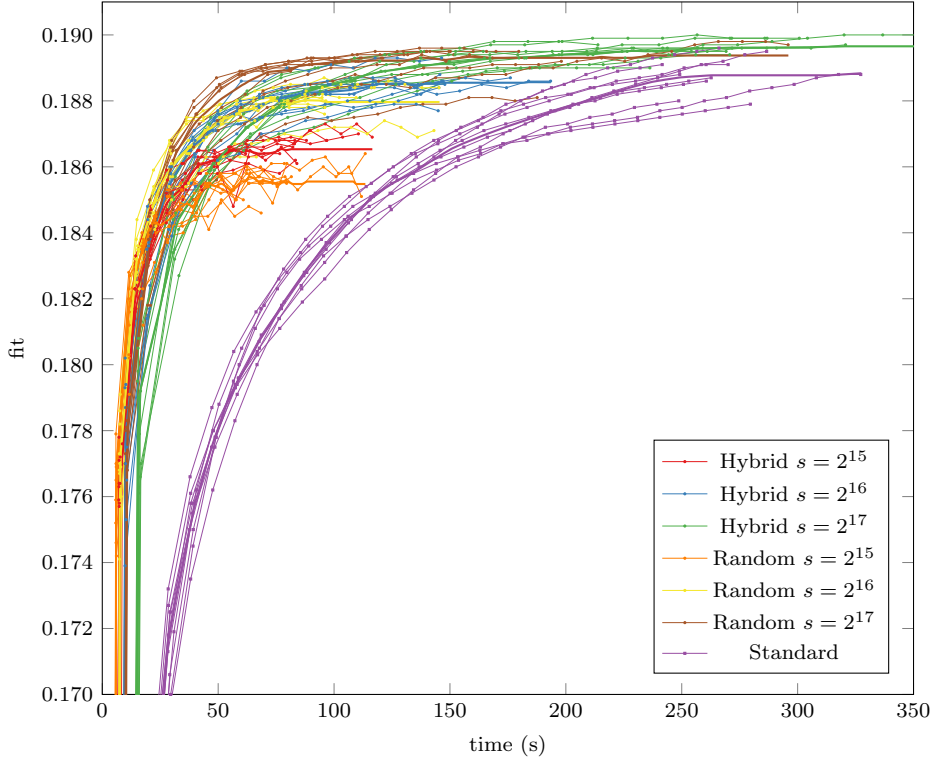
\begin{figure}[h]
  \centering
  \pgfplotsset{
    height=4.25in,
    width=0.95\textwidth,
    xlabel = time (s),
    ylabel = fit,
    ymin = 0.17, ymax = 0.191,
    xmin=0,xmax = 350,
    yticklabel style={%
      /pgf/number format/.cd,
      fixed,
      fixed zerofill,
      precision=3,
    },
    unbounded coords=jump,
    legend pos = south east,
  }
  \begin{tikzpicture}
    \begin{axis}[font=\footnotesize,    
      cycle list={
          {index of colormap=0 of Set1-8,mark=*},
          {index of colormap=1 of Set1-8,mark=*},
          {index of colormap=2 of Set1-8,mark=*},
        {index of colormap=4 of Set1-8,mark=*},
        {yellow!90!gray, mark=*},
        {index of colormap=6 of Set1-8,mark=*},
        {index of colormap=3 of Set1-8,mark=square*}}
      ]
      \foreach \i in {1,...,10} {
        \foreach \m in {1,2,3,5,6,7,4} {        
          \addplot+[mark size=.4pt, very thin] table [x=M\m_R\i_T, y=M\m_R\i_F]{\uberraw};
        }
      }
      \legend{Hybrid $s=2^{15}$,Hybrid $s=2^{16}$,Hybrid $s=2^{17}$,Random $s=2^{15}$,Random $s=2^{16}$,Random $s=2^{17}$,Standard}
      \foreach \m in {1,2,3,5,6,7,4} {
        \addplot+[thick,mark=none] table [x=TIME, y=M\m_MED]{\uberitp};
      }
    \end{axis}
  \end{tikzpicture}
  \caption{Comparison of \cprandlev (random and hybrid) with varying number of samples $s \in \set{2^{15},2^{16},2^{17}}$ and \cpals (standard)
  to compute a rank $r=25$ CP decomposition of the
  Uber tensor with 3.3 million nonzeros. The dotted lines represent the individual runs with
  markers indicating epoch (5 iterations) for the randomized methods and one iteration for the standard method.
  The solid lines show  medians.}
\label{fig:supp-uber}
\end{figure} %

%
%
%
%
%
%
%
%
%
%
%
%
%
%
%
%
%
%
%
%
%
%
%
%
%
%
%
%
%
%
%
%
%
%
%
%
%
%
%
%
%
%
%
%
%
%
%
%
%
%
%
%
%
%
%
%
%
%
%
%
%
%
%
%
%
%
%
%
%
%
%
%
%
%
%
%
%
%
%
%

%
%
%
%
% --- End Inserted File ---

\section{Initialization via Randomized Range Finder for Enron Tensor}
\label{sec:rrf}
% ---- Inserted File ----
%
%
%

This supplement uses the Enron tensor from FROSTT to illustrate how performance can be improved for some
tensors via randomized range finder (RRF) initialization.
The quality of the
randomized least squares is adversely affected by the norm of the right hand side that is outside the range of the matrix,
which we denote as $\Xmat^{\perp}$.
A random initialization could result in large $\Xmat^{\perp}$,
hurting the performance of the run.  We show that this can be fixed by
simply initializing with a random linear combination of the fibers in
the matricized tensor, a method referred to in the literature as
RRF \cite{HaMaTr11}.
(It can also be that $\Xmat^{\perp}$ is large because the method does not have multilinear structure,
but this is a property of the tensor that would generally
lead to sub par performance of any method on the problem.)

The Enron tensor is of size 6,066 $\times$ 5,699 $\times$ 244,268 $\times$ 1,176
and has 54,202,099 nonzeros.
It is the 4-way log-count\footnote{This tensor has been modified from the raw count tensor provided by FROSTT. Each entry is $\log(c+1)$ where $c$ is the count. Note that the zeros are unchanged. This is a standard weighting in text analysis. The primary effect is that the largest entries are damped.}
tensor of emails comprising 6K senders $\times$ 6K receivers $\times$ 244K words $\times$ 1K days.
For each of the run of \cprandlev,
we used hybrid sampling with $\tau = 1/s$ and an estimated fit as described in \cref{sec:estimating-fit} with  $s_{\text{fit}} = 2^{25}$ stratified samples, evenly divided between zeros and nonzeros.
The tensor elements were only sampled once and shared across all epochs and runs for consistency of reporting.
For both methods, the stopping tolerance was $10^{-4}$, and the experiments were run on computational environment (A).

\Cref{fig:enron} shows the difference between runs with a random
initialization and runs initialized via RRF on the Enron tensor.  As before, the random initialization draws from a standard
Gaussian for each element of the factor matrix. Runs initialized via
RRF formed the initial factor matrix from a random linear combination
of the matricized fibers.  This was done by first drawing $s_{\text{init}}$ fibers
uniformly from the nonzero fibers of the matricized tensor, in this
case using $s_{\text{init}} = 10^5$. As \cprandlev already forms the linear
indices of elements along each mode unfolding as a preprocessing
step, sampling and extracting the fibers is an efficient computation.
These are then multiplied by a random Gaussian matrix
$\Mx{\Omega} \in \R^{s_{\text{init}} \times r}$ in order to form a random linear
combination of the sampled fibers for each column of the factor
matrix.  By forming the columns of our initialization out of the
columns of the matricized tensor we tend to decrease the magnitude of
$\Xmat^{\perp}$, or the part of $\Xmat$ that is perpendicular to the
column space of our factor matrix.

\pgfplotstabletranspose[input colnames to={IDs}, colnames from={IDs}]{\enrontime}{data-enron-totaltimes.csv}
\pgfplotstabletranspose[input colnames to={IDs}, colnames from={IDs}]{\enronfit}{data-enron-finalfits.csv}
\begin{figure}[h]
  \centering
  \begin{tikzpicture}
  \begin{axis}[
    boxplot/draw direction = x,
    y dir=reverse,
    height=1.75in,width=0.45\textwidth,
    font=\footnotesize,
    xlabel = fit,
    xticklabel style={%
       /pgf/number format/.cd,
           fixed,
           fixed zerofill,
           precision=2,
           },
    cycle list={[indices of colormap={4, 0, 5, 1, 3 of Set1-8}]},
    ytick = {1, 2, 3, 4, 5},
    yticklabel style = {align=center, font=\footnotesize}, %
    yticklabels = {Gaussian $s=2^{18}$, RRF $s=2^{18}$, Gaussian $s=2^{19}$ , RRF $s=2^{19}$, Standard},
    ]
    \foreach \m in {4, 1, 5, 2, 3} {
      \addplot+[boxplot, fill, draw=black] table[y=M\m] {\enronfit};
    }
  \end{axis}
  \begin{axis}[
    boxplot/draw direction = x,
    y dir=reverse,
    height=1.75in,width=0.5\textwidth,
    font=\footnotesize,
    xlabel = time (s),
    xshift=2.0in,
    cycle list={[indices of colormap={4, 0, 5, 1, 3 of Set1-8}]},
    ytick = {1, 2, 3, 4, 5},
    yticklabel style = {align=center, font=\footnotesize}, %
    yticklabels = {,,,,},
    ]
    \foreach \m in {4, 1, 5, 2, 3} {
      \addplot+[boxplot, fill, draw=black] table[y=M\m] {\enrontime};
    }
  \end{axis}
  \end{tikzpicture}
  \caption{Comparison of different methods of initialization for \cprandlev hybrid  ($\tau=1/s$) on the Enron tensor with 54.2 million nonzeros and rank $r = 25$.  Each box plot represents 10 runs, and the experiments show that initializing via the Randomized Range Finder (RRF) with $s_{\text{init}} = 10^5$ provides a significant improvement in fit compared to Gaussian initialization.}
  \label{fig:enron}
\end{figure}
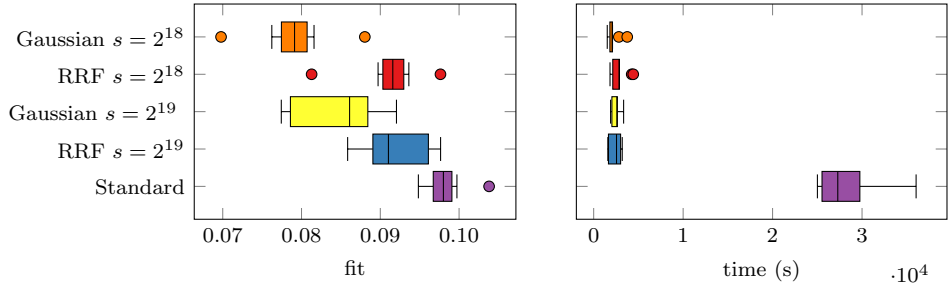 %

The left panel of \cref{fig:enron} shows the fit values across 10 runs for each
initialization method for sample size $s=2^{18}$ and $s=2^{19}$; 10 runs of
\cpals are also included for comparison.  The experiments show that
the RRF greatly improves the fit found by \cprandlev and that
the fit is only comparable to \cpals if the RRF method is used.
The right panel of \cref{fig:enron} shows that the total run time is roughly
the same for either initialization method.  Furthermore, the median
runtime with RRF initialization for $s=2^{18}$ samples is 5.78 times faster and for $s=2^{19}$ samples is 4.39 times
faster than the median runtime for \cpals.

%
%
%
%
% --- End Inserted File ---

\section*{Acknowledgments}%
We would like to thank the referees and editors for useful feedback that has greatly
strengthened the results, the experiments, and the presentation.
In particular, we thank the referee that provided careful feedback on the proof
of \cref{thm:sketching}, enabling reduction of
the number of samples from $s=O(r \log n/(\beta \epsilon^2))$
to $s=O(r / (\beta \epsilon))$.
We would like to thank Daniel Martin for pointing us to the relevant problems of key rank and key enumeration in cryptography and providing us with the relevant references \cite{Glowacz15, Martin15,Poussier16}.
Thanks to Jimmy Peng for feedback on earlier draft.

\bibliographystyle{siamplainmod}
%

%%\\n

\end{document}